\documentclass[13pt,a4paper]{amsart}
\usepackage{amsmath}
\usepackage{amssymb}
\usepackage{amsthm}
\usepackage{graphicx}
\usepackage{mathrsfs}
\usepackage[english]{babel}
\usepackage{fullpage}

\usepackage{tikz}
\usetikzlibrary{matrix,arrows}
\usepgflibrary[arrows]
\usepgflibrary{arrows}
\usetikzlibrary[arrows]
\tikzset{node distance=1.5cm, auto}

\theoremstyle{definition}

\numberwithin{equation}{section}
\theoremstyle{definition}

\newtheorem{algorithm}{Algorithm}[subsection]

\newtheorem{conjecture}[algorithm]{Conjecture}
\newtheorem{corollary}[algorithm]{Corollary}

\newtheorem{definition}[algorithm]{Definition}

\newtheorem{lemma}[algorithm]{Lemma}

\newtheorem{proposition}[algorithm]{Proposition}
\newtheorem{remark}[algorithm]{Remark}

\newtheorem{theorem}[algorithm]{Theorem}

\newcommand{\Spec}{\textrm{Spec}\thinspace}

\usepackage{eucal}

\input cyracc.def
\DeclareFontFamily{U}{russian}{}
\DeclareFontShape{U}{russian}{m}{n}
        { <5><6> wncyr5
          <7><8><9> wncyr7
          <10><10.95><12><14.4><17.28><20.74><24.88> wncyr10 }{}
\DeclareSymbolFont{Russian}{U}{russian}{m}{n}
\DeclareSymbolFontAlphabet{\mathcyr}{Russian}
\makeatletter
\let\@math@cyr\mathcyr
\renewcommand{\mathcyr}[1]{\@math@cyr{\cyracc #1}}
\makeatother
\newcommand{\sha}{\mathcyr{sh}}

\tikzset{
  graphnode/.style = {align=center, inner sep=0pt, scale=0.3, text centered,
    font=\sffamily},
  vi/.style = {graphnode, circle, white, font=\sffamily\bfseries, draw=black,
    fill=black, text width=1em},
  ve/.style = {graphnode, circle, draw=black, 
    text width=1em, thick},
  vx/.style = {graphnode, draw=white,
    minimum width=0em, minimum height=0em},
  vb/.style = {graphnode, diamond, white, draw=black, 
    minimum width=1em, minimum height=1em, thick},
}

\title{Formal weights in Kontsevich's formality construction \\ and multiple zeta values}
\author{Johan Alm}
\address{Stockholm University, Department of Mathematics}
\email{alm@math.su.se}

\begin{document}
\maketitle

\begin{abstract}
We construct a functor \(\boldsymbol{\mathcal{Z}}_{\scriptstyle{\mathrm{ns}}}\) that associates to any dg cooperad \(\mathsf{A}\) of dg commutative algebras (satisfying some conditions) an augmented commutative algebra. When applied to the cohomology operad \(A(M_0^{\delta})\) of Francis Brown's moduli spaces it produces an algebra that formally models the algebra of multiple zeta values. We prove that there is an injection from the graded dual of the Grothendieck-Teichm\"uller Lie algebra into the indecomposables of the algebra \(\boldsymbol{\mathcal{Z}}_{\scriptstyle{\mathrm{ns}}}(\mathsf{coGer})\) associated to the Gerstenhaber cooperad, and that there is a morphism \(\boldsymbol{\mathcal{Z}}_{\scriptstyle{\mathrm{ns}}}(A(M^{\delta}_0))\to \boldsymbol{\mathcal{Z}}_{\scriptstyle{\mathrm{ns}}}(\mathsf{coGer})\) which is surjective on indecomposables.
\end{abstract}
\section*{Introduction}
Let \(k_1,\dots, k_r\) be a sequence of strictly positive integers, with \(k_r\geq 2\). The multiple zeta values (for short, MZVs) are the real numbers
 \[
  \zeta(k_1,\dots,k_r) := \sum_{0<n_1<\dots < n_r} \frac{1}{n_1^{k_1}\cdots n_r^{k_r}}.
 \]
They were studied already by Euler but their extensive study is relatively recent and there are several open ocnjectures regarding them. The deepest result to date concerning these numbers is a theorem due to Francis Brown which says that every period of a mixed Tate motive unramified over the integers is a \(\mathbb{Q}[1/(i2\pi)]\)-linear combination of MZVs.\cite{Brown12} It is relatively easy, using the above displayed series representation, to see that the multiplication of two MZVs is a rational linear combination of MZVs, so they span a subalgebra \(\boldsymbol{\zeta}\) of the real numbers. Most of the open conjectures regarding MZVs concern the structure of this algebra. The number \(k_1+\dots+k_r\) is called the \textit{weight} of the MZV \(\zeta(k_1,\dots,k_r)\) and one of the most fundamental conjectures says that the algebra \(\boldsymbol{\zeta}\) is graded by weight. Define \(\boldsymbol{\zeta}_w\subset\mathbb{R}\) to be the subspace of the real numbers spanned by all MZVs of weight \(w\), \(\boldsymbol{\zeta}_0:=\mathbb{Q}\), and set \(\boldsymbol{\zeta}_{\scriptstyle{\mathrm{formal}}}:=\bigoplus_{w\geq 0}\boldsymbol{\zeta}_w\). (The weight grading conjecture can then be expressed by saying that \(\boldsymbol{\zeta}_{\scriptstyle{\mathrm{formal}}}=\boldsymbol{\zeta}\).) The coefficients of the Knizhnik-Zamolodchikov Drinfel'd associator define a surjection
 \[
  \mathfrak{grt}_1'\oplus \mathbb{Q}\zeta(2) \to Q\boldsymbol{\zeta}_{\scriptstyle{\mathrm{formal}}}
 \]
from the graded dual of the Grothendieck-Teichm\"uller Lie algebra (plus the MZV \(\zeta(2)\)) onto the indecomposable quotient \( Q\boldsymbol{\zeta}_{\scriptstyle{\mathrm{formal}}}=\boldsymbol{\zeta}_{\geq 1}/(\boldsymbol{\zeta}_{\geq 1}\cdot \boldsymbol{\zeta}_{\geq 1})\), see, e.g., \cite{Furusho03}. The map is conjectured to be injective. In this paper we prove some results in a similar vein. First of all we define a functor \(\mathsf{A}\mapsto \boldsymbol{\mathcal{Z}}_{\scriptstyle{\mathrm{ns}}}(\mathsf{A})\) that associates to a cooperad of differential graded (dg, for short) commutative algebras (satisfying some properties) an augmented commutative algebra. Let \(\mathsf{coGer}\) denote the cooperad of Arnol'd algebras, i.e., the cohomology operad of the little disks operad, linearly dual to the operad \(\mathsf{Ger}\) of Gerstenhaber algebras. We prove that there exists a kind of universal \(A_{\infty}\)-structure
 \[
  \boldsymbol{\mathcal{Z}}_{\scriptstyle{\mathrm{ns}}}(\mathsf{coGer})\otimes\mathsf{Ass}_{\infty} \to \boldsymbol{\mathcal{Z}}_{\scriptstyle{\mathrm{ns}}}(\mathsf{coGer})\otimes\mathsf{Ger}
 \]
on Gerstenhaber algebras with coefficients in the algebra \(\boldsymbol{\mathcal{Z}}_{\scriptstyle{\mathrm{ns}}}(\mathsf{coGer})\) and that its coefficients define an \textit{injection}
 \[
  \mathfrak{grt}_1'\oplus \mathbb{Q}\zeta(2) \to Q\boldsymbol{\mathcal{Z}}_{\scriptstyle{\mathrm{ns}}}(\mathsf{coGer}).
 \]
The algebra \(\boldsymbol{\mathcal{Z}}_{\scriptstyle{\mathrm{ns}}}(\mathsf{coGer})\) can be regarded as the algebra of universal coefficients, or \textit{weights}, for Kontsevich's construction in \cite{Kontsevich99} of a formality morphism for the little disks operad, while the universal \(A_{\infty}\) structure plays a r\^ole analogous to that played by the Knizhnik-Zamolodchikov associator (above). The algebra \(\boldsymbol{\mathcal{Z}}_{\scriptstyle{\mathrm{ns}}}(\mathsf{coGer})\) is related to multiple zeta values, as follows. Let \(M_{0,n+1}\) denote the open moduli space of Riemann spheres with \(n+1\) labeled points. Francis Brown introduced in his thesis\cite{Brown09} a kind of partial compactification \(M_{0,n+1}^{\delta}\) of this space, sitting as an intermediary between the open moduli space and the well-known Deligne-Mumford compactifcation. Brown's moduli spaces assemble to a (nonsymmetric) operad, and their cohomology algebras, which we denote \(A(M_{0,n+1}^{\delta})\), assemble to a cooperad of algebras. Brown proved in his thesis that any period integral on one of his partially compactified moduli spaces is a rational linear combination of MZVs. Following this, the algebra \(\boldsymbol{\mathcal{Z}}_{\scriptstyle{\mathrm{ns}}}(A(M_0^{\delta}))\) can be regarded as an algebra of formal multiple zeta values. It is generated by symbols \(I(\alpha)\) corresponding to top-dimensional forms \(\alpha\in A(M_0^{\delta})\), modulo relations corresponding to Stokes' relation and various ``product relations'' that arise naturally from the operadic-geometric structure of the moduli spaces. In particular, there is a surjective morphism of algebras \(\boldsymbol{\mathcal{Z}}_{\scriptstyle{\mathrm{ns}}}(A(M_0^{\delta}))\to \boldsymbol{\zeta}\). We prove that there is a morphism \(\boldsymbol{\mathcal{Z}}_{\scriptstyle{\mathrm{ns}}}(A(M_0^{\delta}))\to \boldsymbol{\mathcal{Z}}_{\scriptstyle{\mathrm{ns}}}(\mathsf{coGer})\), which induces a surjection
 \[
  Q\boldsymbol{\mathcal{Z}}_{\scriptstyle{\mathrm{ns}}}(A(M_0^{\delta}))\to Q\boldsymbol{\mathcal{Z}}_{\scriptstyle{\mathrm{ns}}}(\mathsf{coGer})
 \]
on indecomposables.

\hfill\linebreak
\textit{Many thanks to Dan Petersen and Francis Brown for discussion and helpful criticism.}

\section{Preliminary definitions}
All dg (co)operads are assumed to be (co)augmented, and we will accordingly dispense with the distinction between dg (co)operads and dg pseudo-(co)operads. We otherwise follow the conventions concerning operads adopted in \cite{LodayVallette12}.

The (ordered) set \(\{1,\dots,n\}\) is denoted \([n]\).

If \(V\) is a dg vector space, then we define \(\tau^{\leq d}V\) to be the truncation of \(V\) to all degrees \(\leq d\), i.e., \((\tau^{\leq d}V)^i\) equals \(V^i\) if \(i\leq d\) and is otherwise zero.

\begin{lemma}\label{comonoidal}
The truncation to strictly negative degrees is (non-counitally) symmetric comonoidal.
\end{lemma}

\begin{proof}
The map
 \[
 \tau^{\leq -1}(A\otimes B) \to \tau^{\leq -1}(A)\otimes \tau^{\leq -1}(B)
 \]
is defined by the projections
 \[
 \bigoplus_{i+j=p} A^i\otimes B^j \to \bigoplus_{i+j=p, i,j\leq-1} A^i\otimes B^j
 \]
that discards summands \(A^i\otimes B^j\) not having both \(i\) and \(j\leq -1\). That the two maps
 \[
 \tau^{\leq -1}(A\otimes B\otimes C)\to \tau^{\leq -1}(A)\otimes \tau^{\leq -1}(B\otimes C)\to\tau^{\leq -1}(A)\otimes\tau^{\leq -1}(B)\otimes \tau^{\leq -1}(C)
 \]
and
 \[
 \tau^{\leq -1}(A\otimes B\otimes C)\to \tau^{\leq -1}(A\otimes B)\otimes \tau^{\leq -1}(C)\to \tau^{\leq -1}(A)\otimes \tau^{\leq -1}(B)\otimes \tau^{\leq -1}(C)
 \]
agree is evident: both composites are given by the map that projects out all factors \(A^i\otimes B^j\otimes C^ k\) not having \(i,j,k\leq -1\). The symmetry is likewise clear.
\end{proof}
\section{Algebras of formal weights}
\begin{definition}
A DGCA is a unital dg commutative \(\mathbb{Q}\)-algebra. An \emph{augmentation} of a DGCA  \(A\) is a morphism of DGCAs \(\varepsilon: A\to \mathbb{Q}\). We define \(\boldsymbol{\varepsilon}\mathsf{DGCA}:=(\mathsf{DGCA}\downarrow \mathbb{Q})\) to be the category of augmented DGCAs.

An \emph{ns DGCA cooperad with pullbacks and augmentations} is a conilpotent nonsymmetric cooperad \(\mathsf{A}^{\pi}\) in the category of augmented DGCAs, with \(\mathsf{A}^{\pi}(0)=\mathbb{Q}\). We shall denote the category of DGCA ns cooperads with pullbacks and augmentations by \(\mathsf{coOp}_{\scriptstyle{\mathrm{ns}}}^{\pi}(\boldsymbol{\varepsilon}\mathsf{DGCA})\).
\end{definition}
\begin{remark}
Let \(\mathsf{A}^{\pi}\) be an ns DGCA cooperad with pullbacks and augmentations and let \(\iota:[k]\to [n]\) be an order-preserving injection. Because of the assumption that \(\mathsf{A}^{\pi}(0)=\mathbb{Q}\), there is a unique cooperadic cocomposition \(\pi^*_{\iota} : \mathsf{A}^{\pi}(k) \to \mathsf{A}^{\pi}(n)\), given by a rooted planar \(n\)-corolla to which we graft the basis vector of \(\mathsf{A}(0)\) to all legs not in \(\iota([k])\). We call these cocompositions the \textit{pullback maps} of the ns DGCA cooperad. Let \(\mathsf{A}\) be the cooperad obtained by restricting \(\mathsf{A}^{\pi}\) to only arities \(\geq 1\). The cooperad \(\mathsf{A}^{\pi}\) and its cocomposition can be recovered from \(\mathsf{A}\) and the pullback maps (the latter now considered not as part of the cocomposition, but as an additional piece of structure on \(\mathsf{A}\)). We shall henceforth always regard an ns DGCA cooperad with pullbacks as a cooperad \(\mathsf{A}\) in the category of DGCAs, concentrated in positive arities and equipped with a family of pullback maps \(\pi^*_{\iota}\) satisfying the evident axioms.
\end{remark}
Given any dg Lie coalgebra \(\mathfrak{a}\), we shall let \(C(\mathfrak{a})\) denote the cofree DGCA \(S(\mathfrak{a}[-1])\) with differential \(d+\delta\), where \(d\) is the differential on \(\mathfrak{a}\) extended as a graded derivation and \(\delta\) is a Chevalley-Eilenberg type differential, defined by extending the Lie cobracket \(\mathfrak{a}\to \mathfrak{a}\wedge\mathfrak{a}\) as a derivation. We remark that given any dg cooperad \(\mathsf{C}\), the total complex 
 \[
 \mathrm{Tot}^{\oplus}\mathsf{C}\{-1\} = \bigoplus_{n\geq 1}\mathsf{C}(n)[n-1]
 \]
is a dg Lie coalgebra with cobracket defined by the cocompositions; so the dual dg Lie algebra is 
 \[
  ( \mathrm{Tot}^{\oplus}\mathsf{C}\{-1\} )^* = \mathrm{Tot}\,\mathsf{C}^*\{1\} 
 \]
the usual dg Lie algebra structure on the total space of an operad.
It follows from \ref{comonoidal} that if \(\mathfrak{a}\) is a dg Lie coalgebra, then so is \(\tau^{\leq -1}(\mathfrak{a})\).
\begin{definition}
Let \(\mathsf{A}\) be an ns DGCA cooperad with pullbacks and augmentations. Define the \textit{algebra of formal weights} on \(\mathsf{A}\), which we shall denote \(\mathfrak{Z}_{\scriptstyle{\mathrm{ns}}}(\mathsf{A})\), to be the quotient of \(C(\tau^{\leq -1}(\mathrm{Tot}^{\oplus}\mathsf{A}\{-1\}))\) by the following relations, which we term the \textit{product map relations}: 

Assume \(\alpha\in \mathsf{A}(n)\), \(\alpha'\in\mathsf{A}(n')\), where \(n,n'\geq 2\) and that we are given order-preserving injections \(\iota: [n]\to [n+n'-2]\) and \(\iota': [n']\to [n+n'-2]\). Define \(Sh(\iota,\iota')\) to be the subgroup of the permutation group \(\Sigma_{n+n'-2}\) that consists of only those permuations \(\sigma\) that have the property that both \(\sigma\circ\iota\) and \(\sigma\circ\iota'\) are oder-preserving. Define the shuffle of \(\alpha\) and \(\alpha'\) with respect to \(\iota\) and \(\iota'\) to be the sum
 \[
 \alpha\,\sha_{\iota,\iota'}\, \alpha' := \sum_{\sigma\in Sh(\iota,\iota')}
 (-1)^{\vert\sigma\vert}\pi_{\sigma\circ\iota}^*(\alpha) \wedge \pi_{\sigma\circ\iota'}^*(\alpha'),
 \]
where \(\wedge\) is the product in the algebra \(\mathsf{A}(n+n'-2)\) and \(\vert\sigma\vert\) is the parity of the permutation \(\sigma\). The product map relations are that
 \[
  \alpha \cdot \alpha' = \alpha\, \sha_{\iota,\iota'}\, \alpha'
 \]
for all pairs of order-preserving injections \(\iota\) and \(\iota'\).
\end{definition}

\begin{remark}
The above defined construction is a functor 
 \[
 \mathfrak{Z}_{\scriptstyle{\mathrm{ns}}}:\mathsf{coOp}_{\scriptstyle{\mathrm{ns}}}^{\pi}(\boldsymbol{\varepsilon}\mathsf{DGCA}) \to \boldsymbol{\varepsilon}\mathsf{DGCA}
 \]
from ns DGCA cooperads with pullbacks and augmentations to DGCAs. The algebra unit \(u\) of the algebra \(\mathsf{A}(2)\) is sent to the unit in \(\mathfrak{Z}_{\scriptstyle{\mathrm{ns}}}(\mathsf{A})\) because every pullback \(\pi_{\iota}^*(u)\) has to be a unit, and then the product map relations imply that \(\alpha\cdot u = u\cdot \alpha = 1\) for every \(\alpha\). The augmentation \(\varepsilon: \mathfrak{Z}_{\scriptstyle{\mathrm{ns}}}(\mathsf{A})\to\mathbb{Q}\) is defined by extending the augmentation \(\varepsilon:\mathsf{A}(2)\to\mathbb{Q}\) by zero on all \(\mathsf{A}(n)\), \(n\neq 2\).
\end{remark}

\begin{definition}
We shall refer to \(\mathfrak{Z}_{\scriptstyle{\mathrm{ns}}}\) as the \textit{functor of formal weights}. Analogously, we define the \textit{functor of cohomological weights} 
 \[
 \boldsymbol{\mathcal{Z}}_{\scriptstyle{\mathrm{ns}}}:=H^0\circ \mathfrak{Z}_{\scriptstyle{\mathrm{ns}}}:\mathsf{coOp}_{\scriptstyle{\mathrm{ns}}}^{\pi}(\boldsymbol{\varepsilon}\mathsf{DGCA}) \to \boldsymbol{\varepsilon}\mathsf{CAlg}
 \]
to be the functor into augmented unital commutative algebras that is defined by postcomposing the functor of formal weights with taking the degree zero cohomology. For an augmented dg algebra \(B\) with augmentation ideal \(B^+\), we define the space of indecomposables of the algebra \(B\) to be the quotient \(QB:=B^+/(B^+\cdot B^+)\). Taking indecomposables is functorial and we shall accordingly refer to \(Q\mathfrak{Z}_{\scriptstyle{\mathrm{ns}}}\) as the functor of \textit{indecomposable formal weights}, and to \(Q \boldsymbol{\mathcal{Z}}_{\scriptstyle{\mathrm{ns}}}\) as the functor of \textit{indecomposable cohomological weights}.
\end{definition}

\begin{remark}
An immediate consequence of functoriality is that if \(\mathsf{A}\) carries an action by some group, then the action transfers to an action on the formal weights and on the formal cohomological weights of \(\mathsf{A}\). Furthermore, the long exact sequence associated to the short exact sequence given by a surjection shows that a surjection of complexes concentrated in degrees \(\leq 0\) induces a surjection on degree \(0\) cohomology, so the functor of cohomological weights preserves surjections.
\end{remark}

\begin{lemma}\label{equalityofindecomposables}
There is a natural equality \(H^0(Q\mathfrak{Z}_{\scriptstyle{\mathrm{ns}}})=Q \boldsymbol{\mathcal{Z}}_{\scriptstyle{\mathrm{ns}}}\). 
\end{lemma}
\begin{proof}
A surjection of complexes concentrated in degrees \(\leq 0\) induces a surjection on degree \(0\) cohomology. Applying this to 
 \[
  0\to \mathfrak{Z}_{\scriptstyle{\mathrm{ns}}}^+\cdot\mathfrak{Z}_{\scriptstyle{\mathrm{ns}}}^+ \to \mathfrak{Z}_{\scriptstyle{\mathrm{ns}}}^+
  \to Q\mathfrak{Z}_{\scriptstyle{\mathrm{ns}}}\to 0,
 \]
we can conclude that
 \[
  H^0(Q\mathfrak{Z}_{\scriptstyle{\mathrm{ns}}}) = 
  \mathrm{Coker}(H^0(\mathfrak{Z}_{\scriptstyle{\mathrm{ns}}}^+\cdot \mathfrak{Z}_{\scriptstyle{\mathrm{ns}}}^+)\to H^0(\mathfrak{Z}_{\scriptstyle{\mathrm{ns}}}^+)).
 \]
Applying the same argument to the surjection \(\mathfrak{Z}_{\scriptstyle{\mathrm{ns}}}^+\odot\mathfrak{Z}_{\scriptstyle{\mathrm{ns}}}^+ \to \mathfrak{Z}_{\scriptstyle{\mathrm{ns}}}^+\cdot \mathfrak{Z}_{\scriptstyle{\mathrm{ns}}}^+\) and using the K\"unneth formula proves
 \[
 \mathrm{Im}(H^0(\mathfrak{Z}_{\scriptstyle{\mathrm{ns}}}^+\cdot \mathfrak{Z}_{\scriptstyle{\mathrm{ns}}}^+)\to H^0(\mathfrak{Z}_{\scriptstyle{\mathrm{ns}}}^+))
 =
 \mathrm{Im}(H^0(\mathfrak{Z}_{\scriptstyle{\mathrm{ns}}}^+)\odot H^0(\mathfrak{Z}_{\scriptstyle{\mathrm{ns}}}^+)\to H^0(\mathfrak{Z}_{\scriptstyle{\mathrm{ns}}}^+)).
 \]
\end{proof}
\subsection{Formal weights and deformations}
Given an ns DGCA cooperad \(\mathsf{A}\) with pullbacks and augmentations, set
 \[
 \mathfrak{e}(\mathsf{A}):= \mathrm{Tot}^{\oplus} \mathsf{A}\{-1\},
 \]
as a graded vector space, and equip it with the dg Lie coalgebra structure dual to \(\mathrm{Def}_{\scriptstyle{\mathrm{ns}}}(\mathsf{As}_{\infty}\xrightarrow{\varepsilon} \mathsf{A}^*)\).

Let \(\mathsf{As}\) be the nonsymmetric operad of (nonunital) associative algebras, \(\mathsf{As}(n)=\mathbb{Q}\) for all \(n\geq 1\).
\begin{remark}
Assume that \(\mathsf{A}\) is an ns cooperad in the category of augmented DGCAs (with or without pullbacks). The augmentation \(\varepsilon:\mathsf{A}(2)\to\mathbb{Q}\) defines a morphism of ns dg cooperads
 \[
 \varepsilon: \mathsf{A}\to \mathsf{coAs}.
 \]
Equivalently, it defines a morphism of ns dg operads
 \(
 \varepsilon: \mathsf{As} \to \mathsf{A}^*
 \)
into the dg operad linearly dual to \(\mathsf{A}\). In particular, the augmentation defines a Maurer-Cartan element in
 \[
 \mathrm{Conv}_{\scriptstyle{\mathrm{ns}}}(\mathsf{A},\mathsf{As}\{1\}) = \mathrm{Tot}\,\mathsf{A}^*\{1\}
 = \mathrm{Conv}_{\scriptstyle{\mathrm{ns}}}(\mathsf{coAs}\{-1\},\mathsf{A}^*).
 \]
\end{remark}
\begin{proof}
The Maurer-Cartan equation in \(\mathrm{Conv}_{\scriptstyle{\mathrm{ns}}}(\mathsf{coAs}\{-1\},\mathsf{A}^*)\) is that \(d\varepsilon+\varepsilon\circ\varepsilon=0\). First, \(d\varepsilon=0\) becase the augmentation is a chain map. The equation \(\varepsilon\circ\varepsilon = -\varepsilon\circ_1\varepsilon + \varepsilon \circ_2\varepsilon=0\) (corresponding to the associativity of the binary generator of \(\mathsf{As}\)) holds because both summands equal the augmentation on the algebra \(\mathsf{A}(3)\), hence cancel each other.
\end{proof}
Let \(B\) be a DGCA and \(\mathfrak{a}\) a dg Lie coalgebra. There is a bijection between morphisms of DGCAs 
 \[
 C(\mathfrak{a}) \to  B
 \]
and Maurer-Cartan elements in the dg Lie algebra \(B\otimes \mathfrak{a}^*\). Equivalently, the dg Lie algebra \(C(\mathfrak{a})\otimes\mathfrak{a}^*\) has a Maurer-Cartan element which is universal, in the sense that any other Maurer-Cartan element is obtained from it by a changing coefficients (applying \(B\otimes_{C(\mathfrak{a})}(\,)\)) to some \(C(\mathfrak{a})\)-algebra \(B\). 

Assume that \(\vartheta\) is a Maurer-Cartan element of \(\mathfrak{a}^*\). We can then form the twisted dg Lie algebra \(\mathfrak{a}^*_{\vartheta}\), by adding the term \([\vartheta,\,]\) to the differential. If \(\chi\) is Maurer-Cartan in \(B\otimes \mathfrak{a}^*_{\vartheta}\), then \(\vartheta+\chi\) is Maurer-Cartan in \(B\otimes \mathfrak{a}^*\). Thus, if \(\mathfrak{a}_{\vartheta}\) is the evident dg Lie coalgebra dual to \(\mathfrak{a}^*_{\vartheta}\), then the dg Lie algebra \(C(\mathfrak{a}_{\vartheta})\otimes \mathfrak{a}^*\) has a Maurer-Cartan element which is universal among Maurer-Cartan elements of the form \(\vartheta+\chi\). Interpreting Maurer-Cartan elements as morphisms of DGCAs this says that there exists a universal factorization \(C(\mathfrak{a})\to C(\mathfrak{a}_{\vartheta})\). 
\begin{remark}
There is a projection \(C(\mathrm{Tot}^{\oplus}\mathsf{A}\{-1\})\to C(\tau^{\leq -1}(\mathrm{Tot}^{\oplus}\mathsf{A}\{-1\}))\), hence a projection \(I:C(\mathrm{Tot}^{\oplus}\mathsf{A}\{-1\})\to\mathfrak{Z}_{\scriptstyle{\mathrm{ns}}}(\mathsf{A})\). It sends the unit \(u\) of the algebra \(\mathsf{A}(2)\) to the unit; hence it can be factored as a composite
 \[
  C(\mathrm{Tot}^{\oplus}\mathsf{A}\{-1\}) \to C(\mathfrak{e}(\mathsf{A}))\stackrel{I^+}{\to} \mathfrak{Z}_{\scriptstyle{\mathrm{ns}}}(\mathsf{A}).
 \]
The Maurer-Cartan element 
 \[
 I\in \mathfrak{Z}_{\scriptstyle{\mathrm{ns}}}(\mathsf{A})\otimes \mathrm{Conv}_{\scriptstyle{\mathrm{ns}}}(\mathsf{coAs}\{-1\},\mathsf{A}^*)
 \]
can be regarded as the universal solution to the Maurer-Cartan equation that, first of all, deforms \(\varepsilon\), and secondly, has coefficients that satisfy the product map relations.

Alternatively, we can look at the Maurer-Cartan element \(I^+=I-\varepsilon\) in
 \(
  \mathfrak{Z}^+_{\scriptstyle{\mathrm{ns}}}(\mathsf{A})\otimes \mathrm{Def}_{\scriptstyle{\mathrm{ns}}}(\mathsf{As}_{\infty}\stackrel{\varepsilon}{\to}\mathsf{A}^*).
 \)
\end{remark}
\begin{lemma}\label{surjectiononindecomposables}
The morphism \(I^+:C(\tau^{\leq -1}\mathfrak{e}) \to \mathfrak{Z}_{\scriptstyle{\mathrm{ns}}}\) defines a natural surjection 
 \[
  H^0Q(I^+): H^1(\tau^{\leq 0}(\mathfrak{e}[-1]))\to Q\boldsymbol{\mathcal{Z}}_{\scriptstyle{\mathrm{ns}}}.
 \]
\end{lemma}
\begin{proof}
For the proof, define \(\mathfrak{e}^+(\mathsf{A})\) to be the kernel of the augmentation \(\mathsf{A}(2)\to\mathbb{Q}\), extended by \(0\) to a map \(\mathfrak{e}\to\mathbb{Q}\). The functor of taking indecomposables preserves surjections, so we get a surjection \(\tau^{\leq 0}(\mathfrak{e}[-1])\to Q\mathfrak{Z}_{\scriptstyle{\mathrm{ns}}}\). Its kernel can be suggestively denoted \(\mathbb{Q}u\oplus\tau^{\leq 0}(\mathfrak{e}^+[-1])\,\sha\,\tau^{\leq 0}(\mathfrak{e}^+[-1])\); it is spanned by the unit \(u\in\mathsf{A}(2)\) and all possible \(\alpha\, \sha_{\iota,\iota'}\, \alpha'\in\mathfrak{e}^+[-1]\) with \(\alpha,\alpha'\in\mathfrak{e}^+[-1]\). Since there are no elements in strictly positive degrees, the long exact sequence associated to
 \[
  0 \to \mathbb{Q}u\oplus\tau^{\leq 0}(\mathfrak{e}^+[-1])\,\sha\,\tau^{\leq 0}(\mathfrak{e}^+[-1]) \to \tau^{\leq 0}(\mathfrak{e}[-1])
  \to Q\mathfrak{Z}_{\scriptstyle{\mathrm{ns}}}
 \]
ends with
 \[
  \dots\to H^0(\tau^{\leq 0}(\mathfrak{e}[-1]))\to H^0(Q\mathfrak{Z}_{\scriptstyle{\mathrm{ns}}})\to 0.
 \]
We then apply \ref{equalityofindecomposables}, which said \(H^0(Q\mathfrak{Z}_{\scriptstyle{\mathrm{ns}}})=Q\boldsymbol{\mathcal{Z}}_{\scriptstyle{\mathrm{ns}}}\).
\end{proof}
\subsection{Motivational examples}
\subsubsection{First example}
Let \(\chi_n\) be the set of all unordered pairs \(\{i,j\}\subset [n]\) such that \(i\) and \(j\) are not consecutive in the cyclic ordering. The ring of functions on the moduli space \(M_{0,n}\) of projective lines with \(n\) marked (ordered) points can be written in the form \(\mathbb{Z}[u_{ij},u_{ij}^{-1}\mid \{i,j\}\in \chi_n]/I_n\), for \(u_{ij}\) the cross-ratio
 \[
 \frac{(z_i-z_{j+1})(z_{i+1}-z_j)}{(z_i-z_j)(z_{i+1}-z_{j+1})},
 \]
and \(I_n\) the ideal of relations satisfied by those cross-ratios. Brown introduced in his thesis \cite{Brown09} a kind of partial compactification
 \[
 M_{0,n}^{\delta}:= \Spec \mathbb{Z}[u_{ij}\mid \{i,j\}\in \chi_n]/I_n
 \]
of the moduli space. It sits as an intermediary \(M_{0,n}\subset M_{0,n}^{\delta} \subset \overline{M}_{0,n}\) between the open moduli space and the Deligne-Mumford compactification as the space obtained by adding only those boundary divisors of the Deligne-Mumford compactification that bound the connected component of the set of real points \(M_{0,n}(\mathbb{R})\) corresponding to having the marked points in the canonical order \(z_1<\dots < z_n\). The closure of this connected component inside \(M_{0,n}^{\delta}(\mathbb{C})\), call it \(K_{n-1}\), is an associaheder of dimension \(n-3\).

The Deligne-Mumford compactification is well known to assemble (for varying \(n\)) into a cyclic operad. Since only some boundary strata are allowed in Brown's compactification, his spaces do not admit a natural permutation action on point labels--however, they do assemble into an ns operad. Define \(A(M_{0,n+1}^{\delta})\) to be the de Rham complex of algebraic forms on \(M_{0,n+1}^{\delta}\) with logarithmic singularities on \(\overline{M}_{0,n+1}\setminus M_{0,n+1}^{\delta}\). It follows from Deligne's \cite{Deligne71} that \(A(M_{0,n+1}^{\delta})\) has trivial differential and that \(A(M_{0,n+1}^{\delta})\to H(M_{0,n+1}^{\delta})\) is injective. The cohomology of the spaces \(M_{0,n+1}^{\delta}\) have a pure Hodge structure and that implies that the map into the cohomology is also surjective. By using pullback along the various point-forgetful projections \(\pi_{\iota}:M_{0,n+1}^{\delta}\to M_{0,k+1}^{\delta}\) we obtain an ns DGCA cooperad \(A(M_0^{\delta})\) with pullbacks and augmentations, isomorphic to \(H(M_{0,n+1}^{\delta})\). We present more details on these moduli spaces and cohomology algebras later in the paper.
\begin{proposition}
Integration on the standard associahedra defines a surjective morphism
 \[
 \mathrm{ev}:\mathfrak{Z}_{\scriptstyle{\mathrm{ns}}}(A(M_0^{\delta})) \to \boldsymbol{\zeta}
 \]
onto the algebra of multiple zeta values.
\end{proposition}
\begin{proof}
Kontsevich's well-known integral representation of multiple zeta values proves that every multiple zeta value arises by an integration as suggested. Brown proved the converse in \cite{Brown09}, i.e., that that the integral of any \(\alpha\in A(M_0^{\delta})^{n-2}\) over \(K_n\) is a linear combination of multiple zeta values. Accordingly, with those two remarkable results taken for granted we must only argue that we have a morphism
 \[
 \mathrm{ev}:\mathfrak{Z}_{\scriptstyle{\mathrm{ns}}}(A(M_0^{\delta})) \to \mathbb{R}.
 \]
The map respects the differential because \(\mathrm{ev}\circ(d+\delta)=0\) is exactly Stokes' relation. (The differential \(d\) is zero, i.e., all forms are closed, and the relations reduce to saying that the sum of the inegrals corresponding to the boundary restrictions of a \(\beta\in A(M_{0,n+1}^{delta})^{n-3}\) is zero.) To see that the product map relations are satisfied one uses the birational embedding
 \[
 \pi_{\iota}\times \pi_{\iota'}:M_{0,n+n'-1}\to M_{0,n+1}\times M_{0,n'+1},
 \]
to transform a product of evaluations to an evaluation, cf.~\cite{Brown09} and \cite{BrownCarrSchneps09} by Brown, Carr and Schneps.
\end{proof}
The algebra \(\boldsymbol{\mathcal{Z}}_{\scriptstyle{\mathrm{ns}}}(A(M_0^{\delta}))\) is a bit unsatisfactory as a model for the algebra of (formal) multiple zeta values. The spaces \(M_{0,n+1}^{\delta}\) carry an action of the dihedral group (preserving the cell \(K_n\)) and the identities between evaluations imposed by this are not present in \(\boldsymbol{\mathcal{Z}}_{\scriptstyle{\mathrm{ns}}}(A(M_0^{\delta}))\). However, there is an obvious way to enforce these relations, defining an algebra \(\boldsymbol{\mathcal{Z}}_{\scriptstyle{\mathrm{dih}}}(A(M_0^{\delta}))\), and one could hope that this smaller algebra is isomorphic to the algebra of (formal) multiple zeta values. A similar idea, which very much inspired the present paper, was proposed in \cite{BrownCarrSchneps09}.
\subsubsection{Second example}
Our second example arises from Kontsevich's proof in \cite{Kontsevich99} of the formality of the operad of little disks. Define
 \[
 C_n(\mathbb{C}):=(\mathbb{C}^n\setminus diagonals)/\mathbb{C}\rtimes\mathbb{R}_{>0}
 \]
to be the space of configurations of points in the plane modulo translations and positive dilations. It has a well-known (real) Fulton-MacPherson compactifcation \(\overline{C}_n(\mathbb{C})\), with the nice properties that, firstly, the compactified space is homotopy equivalent to the uncompactified space and, secondly, the boundary inclusions assemble to define the structure of a DGCA cooperad with pullbacks and augmentations on the cohomologies \(H(\overline{C}_n(\mathbb{C}))\) (for varying \(n\)). At the heart of Kontsevich's proof that the operad of little disks is formal is a certain dg operad denoted \(\mathsf{Graphs}\), quasi-isomorphic to the homology operad of \(\overline{C}(\mathbb{C})\). It can be defined as follows.

Let \(\mathsf{g}(n)\) be the graded commutative algebra generated by the \(n(n-1)/2\) degree \(-1\) variables \(e_{ij}=e_{ji}\) (\(1\leq i< j\leq n\)), and let \(\mathcal{T}\) be the graded commutative algebra generated by \(d\) degree \(0\) variables \(x^a\) (\(1\leq a\leq d\)) and \(d\) degree \(1\) variables \(\eta_b\) (\(1\leq b\leq d\)). There is a morphism of dg vector spaces
 \[
  \mathsf{g}(n) \to \mathsf{End}\langle \mathcal{T}\rangle(n) = \mathrm{Map}(\mathcal{T}^{\otimes n},\mathcal{T}),
 \]
sending a generator \(e_{ij}\) to the polydifferential operator
 \[
  \sum_{a=1}^d\bigl( \frac{\partial}{\partial\eta_a^i}\frac{\partial}{\partial x^a_j}+\frac{\partial}{\partial\eta_a^j}\frac{\partial}{\partial x^a_i} \bigr).
 \]
Here \(\partial/\partial x^a_i\) acts as \(\partial/\partial x^a\) on the \(i\)th factor of \(\mathcal{T}^{\otimes n}\). One can via these maps lift the operad structure on \(\mathsf{End}\langle \mathcal{T} \rangle\) to an operad structure on the collection \(\mathsf{g}=\{\mathsf{g}(n)\}\). For example, \(e_{ij}\circ_i e_{kl}=e_{ik}e_{kl}+e_{il}e_{kl}\). To every monomial \(M\in \mathsf{g}\) we associate a graph \(\Gamma\) with set of vertices \([n]\), no legs, and an edge connecting the vertices \(i\) and \(j\) for every generator \(e_{ij}\) appearing in \(M\). Since the generators \(e_{ij}\) have degree \(-1\) we need to order the set of edges of \(\Gamma\) up to an even permutation in order to be able to recover \(M\) from \(\Gamma\). Moreover, note that the degrees imply \(e_{ij}e_{ij}=0\), so \(\Gamma\) cannot contain a double edge. Thus, elements of \(\mathsf{g}\) can be pictured as linear combinations of certain graphs (with some extra data). Let \(\mathsf{G}(n)\) be the degree-completion of the polynomial ring \(\mathsf{g}(n)\) (pictorically, this means allowing formal sums of graphs) and define
 \[
  \mathrm{Tw}\,\mathsf{G}(n) := \prod_{k\geq 0} \mathsf{G}(k+n)_{\Sigma_k}[-2k].
 \]
Elements of \(\mathrm{Tw}\,\mathsf{G}(n)\) can be pictured as formal sums of graphs with two types of vertices: \(n\) white labelled vertices and some number \(k\geq 0\) of black unlabelled vertices. The operadic structure on \(\mathsf{G}\) defines an operadic structure on \(\mathrm{Tw}\,\mathsf{G}\) by only allowing composition in the inputs corresponding to white vertices. Composition in the unlabelled/symmetrized black inputs is not part of the operadic composition but instead interpreted as defining a right action \(\bullet\) of the dg Lie algebra
 \[
  \mathrm{Def}(\mathsf{Lie}\{-1\}_{\infty}\to \mathsf{G})=\prod_{\ell\geq 0} \mathsf{G}(k)[2-2\ell],
 \]
by operadic derivations, c.f.~\cite{Alm13,DolgushevWillwacher12} for the general theory at work here. Here \(\mathsf{Lie}\{-1\}_{\infty}\to \mathsf{G}\) is the morphism sending the binary bracket to \(e_{12}\in\mathsf{G}(2)\). Regard elements of this deformation complex as formal sums of graphs \(\gamma\) with black unlabelled vertices. The action \(\Gamma \mapsto \Gamma \bullet\gamma\) by a graph \(\gamma\) with \(\ell\) vertices on a graph \(\Gamma\in \mathrm{Tw}\,\mathsf{G}(n)\) with \(k\) black vertices is given by suitably symmetrizing the operad composition \(\circ_1:\mathsf{G}(k+n)\otimes \mathsf{G}(\ell) \to \mathsf{G}(k+\ell-1 + n), \Gamma\otimes\gamma \to \Gamma\circ_1\gamma\) to produce a (sum of) graph(s) with \(n\) white vertices and \(k+\ell-1\) black vertices. Using this action we equip \(\mathrm{Tw}\,\mathsf{G}(n)\) with the edge-insertion differential
 \[
\partial \Gamma =  [\,
\begin{tikzpicture}[baseline=-1.45ex,shorten >=0pt,auto,node distance=0.9cm]
 \node[vi] (b) {};
 \node[ve] (w) [below of=b] {}; 
   \path[every node/.style={font=\sffamily\small}]
    (b) edge (w);  
\end{tikzpicture}
\,, \Gamma]  
+ \Gamma \bullet
 \begin{tikzpicture}[baseline=-1.45ex,shorten >=0pt,auto,node distance=0.9cm]
 \node[vi] (b1) {};
 \node[vi] (b2) [below of=b1] {}; 
   \path[every node/.style={font=\sffamily\small}]
    (b) edge (w);  
\end{tikzpicture}\,.
 \] 
The second term uses the action just defined. The first term uses the usual action of the commutator Lie algebra of the space of unary operations (existing for any dg operad). This differential makes \(\mathrm{Tw}\,\mathsf{G}\) a dg operad. The map \(\mathsf{Lie}\{-1\}\to\mathsf{G}\) gives the algebra \(\mathcal{T}[1]\) a Lie algebra structure, and the operad \(\mathrm{Tw}\,\mathsf{G}\) has the same relationship to the operad of Chevalley-Eilenberg complexes \(C(\mathcal{T}[1],\mathsf{End}\langle \mathcal{T}\rangle)\) as the operad \(\mathsf{G}\) has to \(\mathsf{End}\langle \mathcal{T}\rangle\). Let \(\mathsf{Graphs}(n)\) be the subspace of \(\mathrm{Tw}\,\mathsf{G}(n)\) spanned by formal sums of graphs with (i) no connected component containing only black vertices and (ii) all black vertices at least trivalent. One may check that these subspaces constitute a dg suboperad \(\mathsf{Graphs}\). Define \(\mathsf{coGraphs}\) to be the finitely dual dg cooperad, meaning \(\mathsf{coGraphs}(n)^*=\mathsf{Graphs}(n)\) (so elements of \(\mathsf{coGraphs}\) correspond to \textit{finite} sums of graphs).

For a graph \(\Gamma\) we define the corresponding internal graph, to be denoted \(\Gamma^{\scriptstyle{\mathrm{int}}}\), to be the graph with legs obtained by deleting the white vertices and turning edges previously connected to white vertices into legs. Say that a graph \(\Gamma\) is internally connected if either (i) it consists of a single edge connecting two white vertices or (ii) it has no edge connecting two
white vertices and the associated internal graph \(\Gamma^{\scriptstyle{\mathrm{int}}}\) is connected.

Define \(\mathsf{ICG}'(n)[-1]\subset\mathsf{coGraphs}(n)\) to be the subspace spanned by all graphs that are internally connected. Clearly,
 \[
  \mathsf{coGraphs}(n) = S(\mathsf{ICG}'(n)[-1]),
 \]
since any graph can be regarded as a superposition of internally connected graphs. The differential \(\partial\) is a coderivation of the free graded commutative algebra \(S(\mathsf{ICG}[1])\) and the cooperadic compositions likewise respect the algebra structure. Thus \(\mathsf{coGraphs}\) is a cooperad of dg commutative algebras. For every \(n\geq 2\), Kontsevich wrote down a quasi-isomorphism of DGCAs
 \[
  \vartheta : \mathsf{coGraphs}(n) \to \Omega(\overline{C}_n(\mathbb{C})),\, \Gamma\mapsto \vartheta_{AT}^{\Gamma}
 \]
onto the de Rham algebra of (piecewise semialgebraic) differential forms. The cells \(K_n\) defined by having all \(n\) points in a configuration on a line parallell to the real axis are associahedra and simply adding white vertices (and suitably relabeling them) defines pullback maps \(\pi_j^* :\mathsf{coGraphs}(n)\to \mathsf{coGraphs}(n+1)\), so Kontsevich's construction can be turned into the the following proposition:
\begin{proposition}
Integration of Kontsevich's forms \(\vartheta^{\Gamma}\) over the associahedra \(K_n\) defines a morphism of dg algebras
 \(
  \mathrm{ev}_{\scriptstyle{\vartheta}} : \mathfrak{Z}_{\scriptstyle{\mathrm{ns}}}(\mathsf{coGraphs}) \to \mathbb{R}.
 \)
\end{proposition}
The proof is again little more than an application of Stokes' theorem. Kontsevich remarked already in the paper \cite{Kontsevich99}, where he introduced the numbers \(\mathrm{ev}_{\scriptstyle{\vartheta}}(\Gamma)\) (\(\Gamma\) a graph), that they seemed to be closely related to multiple zeta values.

It is a classical result due to Vladimir Arnol'd that the cohomology \(H(C_n(\mathbb{C}))\) is isomorphic as an algebra to the graded commutative algebra \(\mathsf{coGer}(n)\) freely generated by degree \(1\) elements \(\omega_{ij}=\omega_{ji}\) (\(1\leq i\neq j\leq n\)) modulo the relations \(\omega_{ij}\omega_{jk}+\omega_{jk}\omega_{ki}+\omega_{ki}\omega_{ji}=0\).\cite{Arnold69} On \(\mathbb{C}^n\setminus diagonals\) these generators can be taken as the honest differential forms \(\omega_{ij}=d\/\log(z_j-z_i)\). Kontsevich proved that the morphism \(\mathsf{coGraphs} \to \mathsf{coGer}\) which sends a graph with one or more black vertices to zero and an edge \(e_{ij}\) between white vertices to \(\omega_{ij}\) is a quasi-isomorphism of cooperads of dg commutative algebras. (The form \(\vartheta^{e_{ij}}\) represents the class \(\omega_{ij}\).) The operad dual to \(\mathsf{coGer}\) is the operad \(\mathsf{Ger}\) of Gerstenhaber algebras. A Gerstenhaber algebra is a dg vector space \(F\) with a product \(\mu\) that makes \((F,\mu)\) is a dg commutative algebra, a bracket \([\,,\,]\) that makes \((F[1],[\,,\,])\) a dg Lie algebra and, moreover, the two operations must be compatible in the sense that the adjoint action of the Lie bracket acts by derivations og the commutative product. The morphism mentioned above is dual to a quasiisomorphism \(\mathsf{Ger}\to\mathsf{Graphs}\) that sends the commutative product to the graph \( 
\begin{tikzpicture}[baseline=-0.5ex,shorten >=0pt,auto,node distance=0.9cm]
 \node[ve] (w1) {};
 \node[ve] (w2) [right of=w1] {};   
\end{tikzpicture} \,
\) and the Lie bracket to the graph \(
\begin{tikzpicture}[baseline=-0.5ex,shorten >=0pt,auto,node distance=0.9cm]
 \node[ve] (w1) {};
 \node[ve] (w2) [right of=w1] {};  
 \path[every node/.style={font=\sffamily\small}]
    (w1) edge (w2); 
\end{tikzpicture} \,
\). Note however that the morphism
 \[
  \nu:\mathbb{R}\otimes\mathsf{Ass}_{\infty}\to\mathbb{R}\otimes\mathsf{Graphs}
 \]
that corresponds to \(\mathrm{ev}:\mathfrak{Z}_{\scriptstyle{\mathrm{ns}}}(\mathsf{coGraphs}) \to \mathbb{R}\) is \textit{not} equal to the canonical composite
 \[
  \mathsf{Ass}_{\infty} \to \mathsf{Com}\to\mathsf{Ger}\to\mathsf{Graphs}.
 \]
The binary product is the same, \(\nu_2=
\begin{tikzpicture}[baseline=-0.5ex,shorten >=0pt,auto,node distance=0.9cm]
 \node[ve] (w1) {};
 \node[ve] (w2) [right of=w1] {};   
\end{tikzpicture} \,,
\) but there are also higher terms involved, the simplest of which is \(\frac{1}{24}\begin{tikzpicture}[baseline=-1.45ex,shorten >=0pt,auto,node distance=0.9cm]
 \node[vi] (b) {};
 \node[ve] (w1) [below left of=b] {};
 \node[ve] (w2) [below of=b] {}; 
 \node[ve] (w3) [below right of=b] {}; 
   \path[every node/.style={font=\sffamily\small}]
    (b) edge (w1)
    (b) edge (w2)  
    (b) edge (w3);  
\end{tikzpicture}\,\) (a term in \(\nu_3\)). It is straight-forward to check by hand that \(\begin{tikzpicture}[baseline=-1.45ex,shorten >=0pt,auto,node distance=0.9cm]
 \node[vi] (b) {};
 \node[ve] (w1) [below left of=b] {};
 \node[ve] (w2) [below of=b] {}; 
 \node[ve] (w3) [below right of=b] {}; 
   \path[every node/.style={font=\sffamily\small}]
    (b) edge (w1)
    (b) edge (w2)  
    (b) edge (w3);  
\end{tikzpicture}\,\) represents a nontrivial degree one cohomology class in the deformation complex of the canonical \(\mathsf{Ass}_{\infty}\to\mathsf{Graphs}\). In slightly more detail, \(\mathrm{Def}(\mathsf{Ass}_{\infty}\to\mathsf{Graphs})\) has two differentials: the differential given by the internal differential \(\partial\) on \(\mathsf{Graphs}\) and a term \(d_H=[\,\begin{tikzpicture}[baseline=-0.5ex,shorten >=0pt,auto,node distance=0.9cm]
 \node[ve] (w1) {};
 \node[ve] (w2) [right of=w1] {};   
\end{tikzpicture} \,,\,]\). The cocycle \(\begin{tikzpicture}[baseline=-1.45ex,shorten >=0pt,auto,node distance=0.9cm]
 \node[vi] (b) {};
 \node[ve] (w1) [below left of=b] {};
 \node[ve] (w2) [below of=b] {}; 
 \node[ve] (w3) [below right of=b] {}; 
   \path[every node/.style={font=\sffamily\small}]
    (b) edge (w1)
    (b) edge (w2)  
    (b) edge (w3);  
\end{tikzpicture}\,\) is cohomologous to the cocycle \(\,\begin{tikzpicture}[baseline=-0.25ex,shorten >=0pt,auto,node distance=0.8cm]
 \node[ve] (w1) {};
 \node[ve] (w2) [right of=w1] {};
 \node[ve] (w3) [right of=w2] {};
 \node[ve] (w4) [right of=w3] {};  
   \path[every node/.style={font=\sffamily\small}]
    (w1) edge[out=75, in=105] (w3)
    (w2) edge[out=75, in=105] (w4); 
\end{tikzpicture}\,\), which is \(d_H\)-exact but cannot possibly be \(\partial\)-exact (since \(\partial\) adds black vertices). Thus the representation \(\nu\) is truly exotic, in the sense that it is a homotopy nontrivial deformation of the standard representation. See \cite{Alm13} for a lengthy discussion of this structure. The quasiisomorphism \(\mathsf{Ger}\to\mathsf{Graphs}\) implies that every Gerstenhaber algebra should have an exotic \(A_{\infty}\) structure homotopical to \(\nu\). The Gerstenhaber structure on the polyvector fields on some affine (graded) space with a chosen Poisson structure factors through a representation of \(\mathsf{Graphs}\), so for all such Gerstenhaber algebras the exotic structure can be written down explicitly. (This class of Gerstenhaber algebras includes Chevalley-Eilenberg cochain complexes \(C(\mathfrak{g},S(\mathfrak{g}))\) of symmetric algebras on a finite-dimensional Lie algebra \(\mathfrak{g}\), since they can be regarded as complexes of polyvector fields on the affine Poisson manifold \(\mathfrak{g}^*\).) Explicitly extending the exotic \(A_{\infty}\) structure to the class of all Gerstenhaber algebras would be solved by finding a morphism from \(\mathfrak{Z}_{\scriptstyle{\mathrm{ns}}}(\mathsf{coGer})\) to \(\mathfrak{Z}_{\scriptstyle{\mathrm{ns}}}(\mathsf{coGraphs})\).
\section{The Grothendieck-Teichm\"uller Lie algebra and formal weights}
\subsection{Preliminaries from the literature}
Recall the Lie algebra \(\mathfrak{t}_n\) of infinitesimal braids on \(n\) strands. It has generators \(t_{ij}=t_{ji}\), \(1\leq i\neq j\leq n\), and relations
 \[
  [t_{ij}, t_{kl}]=0=[t_{ij}+t_{jk},t_{ik}]
 \]
if \(\{i,j\}\cap \{k,l\}=\emptyset\). It has a length-grading \(\bigoplus_{d\geq 0} \mathfrak{t}_n^{(d)}\) by the number of brackets, and we let 
 \[
 \mathfrak{t}'_n := \bigoplus_{d\geq 0} (\mathfrak{t}_n^{(d)})^* 
 \]
be the graded dual. It is naturally a Lie coalgebra, and its linear dual \((\mathfrak{t}'_n)^*\) is the length-completion \(\hat{\mathfrak{t}}_n\). The collection \(\mathfrak{t}=\{\mathfrak{t}_n\}_{n\geq 2}\) forms an operad in the category of Lie algebras (with direct sum as tensor product), in a way that we can describe as follows, paraphrasing Dimitry Tamarkin's \cite{Tamarkin02}.
\begin{lemma}
Given any function \(f:[k]\to [n]\), the mapping
 \[
  t_{ij} \mapsto \sum_{f\{a,b\}=\{i,j\}} t_{ab}
 \]
on generators defines a morphism \(\pi_{f} : \mathfrak{t}_n \to \mathfrak{t}_k\) of Lie algebras. Analogously, if \(g:[n]\to [m]\) is an injection, then \(t_{ij}\mapsto t_{g(i)g(j)}\) is a morphism of Lie algebras \(\kappa_g: \mathfrak{t}_n \to \mathfrak{t}_m\).
\end{lemma}
Now, assume \(n,n'\geq 2\) and \(1\leq i\leq n\). Define \(f:[n+n'-1]\to [n]\) to be the function defined by \(f(s)=s\) for \(s<i\), \(f(s)=i\) for \(i\leq s\leq i+n'-1\) and \(f(s)=s-n'+1\) if \(i+n'\leq s\). Define \(g:[n']\to [n+n'-1]\) to be the function \(g(\ell)=i+\ell-1\).
\begin{proposition}\cite{Tamarkin02}
The morphisms \(\circ_i := \pi_f\oplus \kappa_g : \mathfrak{t}_n\oplus\mathfrak{t}_{n'} \to \mathfrak{t}_{n+n'-1}\) define an operad \(\mathfrak{t}\) in the category of Lie algebras.
\end{proposition}
It is transparent from the above definition of the operadic compositions that one can formally allow \(n'=0\) and then deduce that if \(f:[n-1]\to [n]\) is the injection that misses \(i\in [n]\), the morphism \(\pi_f: \mathfrak{t}_n \to \mathfrak{t}_{n-1}\) is compatible with operad composition and can be regarded as an insertion of constants into the \(i\)th input. Dualizing, we obtain the following statement:
\begin{corollary}
Let \(\iota:[k]\to [n]\) be any injection. The maps \(\pi^*_{\iota}\) dual to the \(\pi_{\iota}\)'s give \(C(\mathfrak{t}')\) the structure of a DGCA ns cooperad with pullbacks and augmentations.
\end{corollary}
\begin{remark}
The cohomology of \(C(\mathfrak{t}')\) is the Gerstenhaber cooperad \(\mathsf{coGer}\), via the identification that takes \(t^*_{ij}\) to \(\omega_{ij}\) and words of length more than one to zero. Equivalently put, the morphism \(\mathsf{Ger}\to C(\hat{\mathfrak{t}})\), \(C(\hat{\mathfrak{t}}):=C(\mathfrak{t}')^*\), that sends the product to \(1\in C(\hat{\mathfrak{t}}_2)\) and the bracket to \(t_{12}\in C(\hat{\mathfrak{t}}_2)\) is a quasiisomorphism.\cite{Kontsevich99} Thomas Willwacher and Pavol \v{S}evera proved the following more detailed form of this quasi-isomorphism in \cite{SeveraWillwacher11}. Let \(\mathsf{ICG}:=(\mathsf{ICG}')^*\). The differential on the dg coalgebra \(\mathsf{Graphs}=\hat{S}(\mathsf{ICG}[1])\) is an \(L_{\infty}\) structure on \(\mathsf{ICG}\). Define the truncation \(\mathsf{TCG}\subset\mathsf{ICG}\) as follows. \(\mathsf{TCG}^{\leq-1}:= \mathsf{ICG}^{\leq -1}\), \(\mathsf{TCG}^{\geq 1}:=0\) and \(\mathsf{TCG}^0\) is the degree \(0\) cocycles of \(\mathsf{ICG}\).  Then \(H(\mathsf{ICG}(n))^0=\hat{\mathfrak{t}}_n\), with \(t_{ij}\) given as the class of \(e_{ij}\), and projection onto degree \(0\) cohomology classes, respectively the inclusion, is a zig-zag of quasi-isomorphisms of \(L_{\infty}\) algebras \(\hat{\mathfrak{t}}_n\leftarrow \mathsf{TCG}(n)\rightarrow\mathsf{ICG}(n)\), \(n\geq 2\).
\end{remark}
The subspace \(\mathrm{Tot}\,\hat{\mathfrak{t}}\{1\}[1]\) of the deformation complex \(\mathrm{Def}_{\scriptstyle{\mathrm{ns}}}(\mathsf{As}_{\infty}\xrightarrow{\varepsilon} C(\hat{\mathfrak{t}}))\) is closed under the differential and hence forms a subcomplex \((\mathrm{Tot}\,\hat{\mathfrak{t}}\{1\}[1],\partial_{\varepsilon})\). Note that \(H^1(\mathrm{Tot}\,\hat{\mathfrak{t}}\{1\}[1],\partial_{\varepsilon})\) for degree reasons consists of series \(\psi\in \hat{\mathfrak{t}}_3\) satisfying a cocycle condition. (There are no exact cocycles.) The Lie algebra \(\mathfrak{t}_3\) is isomorphic to a sum \(\mathbb{Q}z\oplus \mathfrak{lie}(x,y)\) of a free Lie algebra on two generators and a central element spanning a one-dimensional Abelian Lie algebra. Using this one may realize cocycles as series \(\psi\in\widehat{\mathfrak{lie}}(x,y)\) satisfying the so-called pentagon equation
 \[
 \begin{split}
  0=\psi(t_{12},t_{23})&-\psi(t_{13}+t_{23},t_{34})+\psi(t_{12}+t_{13},t_{24}+t_{34})\\
  &-\psi(t_{12},t_{23}+t_{24})+\psi(t_{23},t_{34})\in \hat{\mathfrak{t}}_4.
 \end{split}
 \]
\begin{theorem}\cite{Furusho10,Tamarkin02,Willwacher10}\label{willwacherfurusho}
The inclusion of the complex \((\mathrm{Tot}\,\hat{\mathfrak{t}}\{1\}[1],\partial_{\varepsilon})\) into \(\mathrm{Def}_{\scriptstyle{\mathrm{ns}}}(\mathsf{As}_{\infty}\xrightarrow{\varepsilon} C(\hat{\mathfrak{t}}))\)
is a quasi-isomorphism. The degree one cohomology is
 \[
  H^1(\mathrm{Tot}\,\hat{\mathfrak{t}}\{1\}[1],\partial_{\varepsilon})=\mathfrak{grt}_1\oplus \mathbb{Q}\cdot [x,y].
 \]
In other words, if \(\psi\in\widehat{\mathfrak{lie}}(x,y)\) satisfies the pentagon equation and does not contain the bracket \([x,y]\), then \(\psi\) lies in the Grothendieck-Teichm\"uller Lie algebra \(\mathfrak{grt}_1\). The degree zero cohomology is one-dimensional, \(H^0(\mathrm{Tot}\hat{\mathfrak{t}}\{1\}[1],\partial_{\varepsilon})=\hat{\mathfrak{t}}_2\), and the cohomology in strictly negative degrees vanishes.
\end{theorem}
\begin{corollary}
\( H^1(\mathrm{Def}_{\scriptstyle{\mathrm{ns}}}(\mathsf{As}_{\infty}\xrightarrow{\varepsilon} \mathsf{Ger}))\cong \mathfrak{grt}_1\oplus \mathbb{Q}\cdot [x,y]\).
\end{corollary}
The length-grading on the Lie algebra \(\mathfrak{t}_n\) defines an additional grading on \(C(\mathfrak{t}_n)\). To be precise, we will say that a Lie word in \(\mathfrak{t}_n\) has length \(d\) if it contains \(d\) nested Lie-brackets of generators \(t_{ij}\), and define the grading on the Chevalley-Eilenberg complex by extending multiplicatively. Cocycles in \(\mathrm{Def}_{\scriptstyle{\mathrm{ns}}}(\mathsf{As}_{\infty}\xrightarrow{\varepsilon} \mathsf{Ger})\) that represent the Grothendieck-Teichm\"uller elements can be deduced from the following lemma.
\begin{lemma}\cite{Kontsevich99,LambrechtsVolic08}
Any \(\partial_{CE}\)-cocycle in \(\mathrm{Def}_{\scriptstyle{\mathrm{ns}}}(\mathsf{As}_{\infty}\xrightarrow{\varepsilon} C(\mathfrak{t}))\) of length \(\geq 1\) is \(\partial_{CE}\)-exact.
\end{lemma}
\begin{corollary}\label{representatives}
Let \(\psi\in \mathfrak{t}_3\) be a degree \(1\) cocycle in \((\mathrm{Tot}\,\mathfrak{t}\{1\}[1],\partial_{\varepsilon})\) of length \(k\). Then for \(3\leq r\leq k+3\) there are \(\varphi_r\in C^{r-1}(\mathfrak{t}_r)\) of length \(k+3-r\) such that
 \[
 \psi = (\partial_{CE}+\partial_{\varepsilon})(\varphi_1+\dots+\varphi_r)-\partial_{\varepsilon}\varphi_r
 \]
in \(\mathrm{Def}_{\scriptstyle{\mathrm{ns}}}(\mathsf{As}_{\infty}\xrightarrow{\varepsilon} C(\mathfrak{t}))\).
\end{corollary}
\begin{proof}
By the lemma there exists \(\varphi_3\) such that \(\psi=\partial_{CE}\varphi_3\). Then note that \(-\partial_{\varepsilon}\varphi_3\) is a \(\partial_{CE}\)-cocycle, so \(-\partial_{\varepsilon}\varphi_3 = \partial_{CE}\varphi_4\) has a solution \(\varphi_4\), again by the lemma. Continue inductively.
\end{proof}
\begin{remark}
The cocycle \(-\partial_{\varepsilon}\varphi_{k+3}\) has length \(0\), hence is a product of generators \(t_{ij}\) and one can proceed to try to find its preimage in \(\mathrm{Def}_{\scriptstyle{\mathrm{ns}}}(\mathsf{As}_{\infty}\xrightarrow{\varepsilon} \mathsf{Ger})\) under \(\mathsf{Ger}\to C(\mathfrak{t})\). Conversely, if one has a class in the deformation complex of the morphism into the Gerstenhaber operad, then one finds a corresponding cocycle representative in \((\mathrm{Tot}\,\hat{\mathfrak{t}}\{1\}[1],\partial_{\varepsilon})\) by applying the recipe in reverse. The simplest example is the class given by the bracket \([x,y]=[t_{13},t_{23}]\). It is cohomologous to \(-t_{13}\wedge t_{24}\), which is the image of the Gerstenhaber operation \(-[1,3]\wedge [2,4]\).
\end{remark}
\subsection{Injection into the indecomposable cohomological weights}
Let \(CH(\mathfrak{t}')\) be the complex \((\mathrm{Tot}^{\oplus}{\mathfrak{t}'}\{-1\}[-1],\partial_{\varepsilon}')\), dual to \((\mathrm{Tot}\,\hat{\mathfrak{t}}\{1\}[1],\partial_{\varepsilon})\) and note that there is a chain map \(\mathfrak{e}(C(\mathfrak{t}')\to CH(\mathfrak{t}')\). Let \(\mathfrak{grt}'_1\) denote the graded dual of the space of \(\psi\in\mathfrak{lie}(x,y)\) (no degree completion) satisfying the pentagon and not containing \([x,y]\), so \((\mathfrak{grt}_1')^*=\mathfrak{grt}_1\). 
\begin{corollary}
The degree zero cohomology of \((\tau^{\leq -1}CH(\mathfrak{t}'))[-1]\) is \(\mathfrak{grt}_1'\oplus\mathbb{Q}x^{01}\), if we by \(x^{01}\) denote the cohomology class dual to that defined by the bracket \([x,y]\).
\end{corollary}
\begin{proof}
The differential
 \[
 \partial_{\varepsilon}: (\mathrm{Tot}\,\hat{\mathfrak{t}}\{1\}[1])^0=\hat{\mathfrak{t}}_2\to (\mathrm{Tot}\,\hat{\mathfrak{t}}\{1\}[1])^1=\hat{\mathfrak{t}}_3
 \]
is \(0\), so the truncation \((\tau^{\geq 1}(\mathrm{Tot}\,\hat{\mathfrak{t}}\{1\}[1],\partial_{\varepsilon}))[1]\) has the same cohomology as the untruncated complex. Dualizing, we deduce that \((\tau^{\leq -1}CH(\mathfrak{t}'))[-1]\) has the same cohomology as \(CH(\mathfrak{t}')[-1]\). The result is then a consequence of \ref{willwacherfurusho}.
\end{proof}
The goal of this section is to prove that \(I^+\) defines an injection
 \[
  \mathfrak{grt}_1'\oplus\mathbb{Q}x^{01} \to Q \boldsymbol{\mathcal{Z}}_{\scriptstyle{\mathrm{ns}}}(\mathsf{coGer})
 \]
into the indecomposable cohomological weights on the Gerstenhaber cooperad. As a first step, we prove the following.
\begin{lemma}\label{injectivebraids}
The composite
 \[
 \mathfrak{grt}_1'\oplus\mathbb{Q}x^{01}\cong H^0(\mathfrak{e}(C(\mathfrak{t}'))[-1])\to H^0(\tau^{\leq 0}(\mathfrak{e}(C(\mathfrak{t}'))[-1])) \xrightarrow{I^+} 
 Q\boldsymbol{\mathcal{Z}}_{\scriptstyle{\mathrm{ns}}}(C(\mathfrak{t}'))
 \]
is injective.
\end{lemma}
\begin{proof}
Recall that we denote \(\mathfrak{e}^+=\mathfrak{e}/\mathbb{Q}u\subset\mathfrak{e}\) (the subcomplex obtained by removing the unit \(u\in\mathsf{A}(2)\)). Consider the following diagram:
\[
\begin{tikzpicture}[descr/.style={fill=white,inner sep=2.5pt}]
    \matrix (m) [matrix of math nodes, row sep=2em,
    column sep=2.5em, text height=1.5ex, text depth=0.25ex]
    { & \mathfrak{e}^+(C(\mathfrak{t}')[-1])  \\
   \tau^{\leq 0}(\mathfrak{e}^+(C(\mathfrak{t}'))[-1])\,\sha\,\tau^{\leq 0}(\mathfrak{e}^+(C(\mathfrak{t}'))[-1])
    & \tau^{\leq 0}(\mathfrak{e}^+(C(\mathfrak{t}'))[-1]) \\
    & \tau^{\leq 0}(CH(\mathfrak{t}')[-1]) \\ };
    \path[->,font=\scriptsize]
    (m-1-2) edge node[auto] {} (m-2-2)
    (m-2-1) edge node[auto] {} (m-2-2)
    (m-2-2) edge node[auto,swap] {} (m-3-2);
\end{tikzpicture}
\]
Here \(\tau^{\leq 0}(\mathfrak{e}^+(C(\mathfrak{t}'))[-1])\,\sha\,\tau^{\leq 0}(\mathfrak{e}^+(C(\mathfrak{t}'))[-1])\) is the kernel of the projection \(I^+\) from \(\tau^{\leq 0}(\mathfrak{e}^+(C(\mathfrak{t}'))[-1])\) to \(Q\mathfrak{Z}_{\scriptstyle{\mathrm{ns}}}(C(\mathfrak{t}'))\). The vertical composite
 \[
 \mathfrak{e}^+(C(\mathfrak{t}'))[-1] \to \tau^{\leq 0}(CH(\mathfrak{t}')[-1])
 \]
is an isomorphism on degree \(0\) cohomology. The composite
 \[
 \tau^{\leq 0}(\mathfrak{e}^+(C(\mathfrak{t}'))[-1])\,\sha\,\tau^{\leq 0}(\mathfrak{e}^+(C(\mathfrak{t}'))[-1]),
 \to \tau^{\leq 0}(CH(\mathfrak{t}')[-1])
 \]
on the other hand, is zero, already at the level of complexes, because any element \(\alpha\, \sha_{\iota,\iota'}\, \alpha'\) in the kernel of the projection to indecomposables must be of tensor-length \(\geq 2\) (i.e., must lie in \(\bigoplus_n C^{\geq 2}(\mathfrak{t}'_n)[n-2]\)), and the vertical arrow down to \(CH(\mathfrak{t}')\) is projection onto tensor-length \(1\). Taking the long exact sequence associated to the projection onto indecomposables and using the degree-truncation shows that
 \[
 H^0(Q\mathfrak{Z}_{\scriptstyle{\mathrm{ns}}}(C(\mathfrak{t}')))
 = \mathrm{Coker}\bigl(
 H^0(\tau^{\leq 0}(\mathfrak{e}^+(C(\mathfrak{t}'))[-1])\,\sha\,\tau^{\leq 0}(\mathfrak{e}^+(C(\mathfrak{t}'))[-1]))
 \to H^0(\tau^{\leq 0}(\mathfrak{e}^+(C(\mathfrak{t}'))[-1]))\bigr).
 \]
We then apply \ref{equalityofindecomposables}.
\end{proof}
\begin{theorem}
The composite
 \[
 \mathfrak{grt}_1'\oplus\mathbb{Q}x^{01}\cong H^0(\mathfrak{e}(\mathsf{coGer})[-1])\to H^0(\tau^{\leq 0}(\mathfrak{e}(\mathsf{coGer})[-1])) \xrightarrow{I^+} 
 Q\boldsymbol{\mathcal{Z}}_{\scriptstyle{\mathrm{ns}}}(\mathsf{coGer})
 \]
is injective.
\end{theorem}
\begin{proof}
The statement is deduced from the following diagram:
\[
\begin{tikzpicture}[descr/.style={fill=white,inner sep=2.5pt}]
    \matrix (m) [matrix of math nodes, row sep=2em,
    column sep=2.5em, text height=1.5ex, text depth=0.25ex]
    { \tau^{\leq 0}(\mathfrak{e}^+(\mathsf{coGer})[-1])\,\sha\,\tau^{\leq 0}(\mathfrak{e}^+(\mathsf{coGer})[-1])
    & \tau^{\leq 0}(\mathfrak{e}^+(\mathsf{coGer})[-1]) 
    & Q\mathfrak{Z}_{\scriptstyle{\mathrm{ns}}}(\mathsf{coGer}) \\
   \tau^{\leq 0}(\mathfrak{e}^+(C(\mathfrak{t}'))[-1])\,\sha\,\tau^{\leq 0}(\mathfrak{e}^+(C(\mathfrak{t}'))[-1])
    & \tau^{\leq 0}(\mathfrak{e}^+(C(\mathfrak{t}'))[-1]) 
    & Q\mathfrak{Z}_{\scriptstyle{\mathrm{ns}}}(C(\mathfrak{t}')) \\
    & \tau^{\leq 0}(CH(\mathfrak{t}')[-1]) & \\ };
    \path[->,font=\scriptsize]
    (m-1-1) edge node[auto] {} (m-1-2)
    (m-1-2) edge node[auto] {} (m-1-3)
    (m-2-1) edge node[auto,swap] {} (m-2-2)
            edge node[auto,swap] {} (m-1-1)
    (m-2-2) edge node[auto,swap] {} (m-2-3)
            edge node[auto,swap] {} (m-1-2)
            edge node[auto,swap] {} (m-3-2)
    (m-2-3) edge node[auto] {} (m-1-3);
\end{tikzpicture}
\]
Take a degree \(0\) element \(\gamma\) in the image of \(\mathfrak{e}(\mathsf{coGer})[-1]\to \tau^{\leq 0}(\mathfrak{e}(\mathsf{coGer})[-1])\), representing some nontrivial cohomology class in \(\mathfrak{e}(\mathsf{coGer})[-1]\). We may without harm identify it as an element in 
 \[
 \bigoplus_n \mathbb{Q}[t_{ij}^*][n-2]\subset \mathfrak{e}(C(\mathfrak{t}'))[-1],
 \]
the sum of the free graded commutative algebras generated by the degree \(1\) generators \(t^*_{ij}\) dual to the generators of the Lie algebras \(\mathfrak{t}_n\). The Arnol'd relations define projections \(\mathbb{Q}[t_{ij}^*]\to\mathsf{coGer}(n)\), sending \(t^*_{ij}\) to \(\omega_{ij}\). In \ref{arnoldbasis} we give an explicit recipe for inverting these projections, by writing elements of the Arnol'd algebra in a preferred basis. This cocycle \(\gamma\) is in the kernel of the projection down to \(\tau^{\leq 0}(CH(\mathfrak{t}')[-1])\), but must by \ref{willwacherfurusho} be cohomologous to some cocycle that is not, i.e., we can write \(\gamma=\alpha+d\beta\) (\(d=d_{CE}+d_{\varepsilon}\)) where \(\alpha\) has nontrivial projection to \(\tau^{\leq 0}(CH(\mathfrak{t}')[-1])\). (Compare with \ref{representatives}.) Assume to get a contradiction that \(\gamma=\gamma_1\,\sha_{\iota,\iota'}\,\gamma_2\) is a nontrivial product. We can repeat the argument and write \(\gamma_i=\alpha_i+d\beta_i\). But then
 \[
 \gamma = \alpha_1\,\sha_{\iota,\iota'}\,\alpha_2+d(\beta_1\,\sha_{\iota,\iota'}\,\alpha_2+d\beta_1\,\sha_{\iota,\iota'}\,\alpha_2+\beta_1\,\sha_{\iota,\iota'}\,d\beta_2),
 \]
which implies that the cohomology class defined by \(\alpha\) is the same as that defined by \(\alpha_1\,\sha_{\iota,\iota'}\,\alpha_2\). However, by the previous lemma, \ref{injectivebraids}, this is impossible.
\end{proof}
\subsection{Relation to period integrals on Brown's moduli spaces}
We begin this section with some recollections concerning Brown's moduli spaces \(M_{0,n+1}^{\delta}\), mostly borrowing from \cite{Brown09}. 

Define the open moduli space of \(n\)-pointed genus zero curves as the quotient manifold
 \[
 M_{0,n+1} := ((\mathbb{CP}^1)^{n+1}\setminus diagonals)/ PSL_2(\mathbb{C}).
 \]
It is an algebraic variety and the ring of functions has the following presentation. Define \(\chi_{n+1}\) to be the set of unordered pairs \(\{i,j\}\) of indices \(i,j\in [n+1]\) that are not consecutive modulo \(n+1\). We shall follow Brown and refer to \(\chi_{n+1}\) as the set of \textit{chords} on \([n+1]\), and to an element of this set as a chord. Given a chord \(\{i,j\}\in \chi_{n+1}\), let \(u_{ij}\) denote the cross-ratio
 \[
 u_{ij} := [i\,i+1\mid j+1\,j] := \frac{(z_i-z_{j+1})(z_{i+1}-z_j)}{(z_i-z_j)(z_{i+1}-z_{j+1})}.
 \]
It is well-defined as a function on \(M_{0,n+1}\). Considering \([n+1]\) as cyclically ordered in the natural way, any chord \(\{i,j\}\) will partition \([n+1]\setminus \{i,j\}\) into two connected components. Say that two chords \(\{i,j\}\) and \(\{k,l\}\) \textit{cross} if \(k\) and \(l\) belong two different connected components in the partition defined by \(\{i,j\}\). (This is obviously a symmetric condition in the sense that this is true if and only if \(i\) and \(j\) lie in different connected components of the partition defined by \(\{k,l\}\).) Given a subset \(A\subset \chi_{n+1}\), let \(A^{\bot}\) denote the set of chords that cross every chord in \(A\), and say that two subsets \(A,B\subset\chi_{n+1}\) \textit{cross completely} if \(A^{\bot}=B\) and \(B^{\bot}=A\). One can then argue that the ring of functions on the moduli space is
 \[
 \mathcal{O}(M_{0,n+1})= \mathbb{Q}[u_{ij}, u_{ij}^{-1} \mid \{i,j\}\in\chi_{n+1}] / \langle R\rangle,
 \]
where \(R\) is the spanned by all elements
 \[
 1-\prod_{\{i,j\}\in A} u_{ij}-\prod_{\{k,l\}\in B} u_{kl},
 \]
labeled by pairs of completely crossing subsets \(A,B\subset\chi_{n+1}\). This leads to a description of the cohomology algebra \(H(M_{0,n+1})\) as the graded commutative algebra generated by the degree \(1\) elements 
 \[
\alpha_{ij}:=d\log u_{ij},\; \{i,j\}\in\chi_{n+1},
 \]
modulo relations saying that
 \[
 \bigl( \sum_{\{i,j\}\in A} \alpha_{ij} \bigr)\bigl(\sum_{\{k,l\}\in B} \alpha_{kl}\bigr)=0
 \]
for all pairs of completely crossing subsets \(A,B\subset\chi_{n+1}\). By fixing the point \(z_{n+1}\) to lie at \(\infty\), one obtains a presentation
 \[
 M_{0,n+1}=(\mathbb{C}^n\setminus diagonals)/\mathbb{C}\rtimes\mathbb{C}^{\times}
 \]
of the open moduli space as the base space of a circle fibration \(C_n(\mathbb{C})\to M_{0,n+1}\), and this leads to an alternative description of the cohomology algebra. Pullback along the fibration defines an inclusion from \(H(M_{0,n+1})\) into the Arnol'd algebra, mapping 
 \[
 \alpha_{ij} \mapsto \omega_{i\, j+1}+\omega_{i+1\, j}-\omega_{i+1\, j+1}-\omega_{ij},
 \]
with the provisio that \(\omega_{i\,n+1}=0\) for all \(i\). The image can be characterized as the kernel of the graded derivation 
 \[
 \iota_v = \sum_{i,j} \frac{\partial}{\partial\omega_{ij}} : \mathsf{coGer}(n)\to \mathsf{coGer}(n)[-1]
 \]
of the Arnol'd algebra. 

Brown's moduli space \(M_{0,n+1}^{\delta}\) is the variety
 \[
 M_{0,n+1}^{\delta} := \mathrm{Spec}\,\mathbb{Q}[u_{ij}\mid \{i,j\}\in\chi_{n+1}] /\langle R\rangle,
 \]
where \(R\) is the same set of relations as that defining the open moduli space. It's cohomology algebra, which we will denote \(A(M^{\delta}_{0,n+1})\) is a lot more subtle to describe explicitly than that of the open moduli space, but it can be described as the subalgebra of \(H(M_{0,n+1})\) spanned by those linear combinations of monomials in the \(\alpha_{ij}\)'s that have vanishing residue along all complex codimension one boundary strata \(D\in M_{0,n+1}^{\delta}\setminus M_{0,n+1}\). Each such divisor has the form \(M_{0,n-k+2}^{\delta}\times M_{0,k+1}^{\delta}\), and corresponds to the collapse of all the points \((z_i)_{i\in S}\) in a cyclically consecutive subset \(S\subset [n+1]\) of cardinality \(\#S=k+1\). Inclusions of strata,
 \[
 M_{0,n-k+2}^{\delta}\times M_{0,k+1}^{\delta} \to M_{0,n+1}^{\delta},
 \]
define a nonsymmetric operad structure on \(M^{\delta}_0\), hence a nonsymmetric cooperad structure on \(A(M^{\delta}_0)\). Adding pullbacks along point-forgetting projections defines an ns DGCA cooperad with pullbacks and augmentations.
\begin{lemma}
The inclusions \(A(M^{\delta}_{0,n+1})\subset H(M_{0,n+1})\to \mathsf{coGer}(n)\) are a morphism of ns DGCA cooperads with pullbacks and augmentations.
\end{lemma}
\begin{proof}
One can either do a direct algebraic proof or rely on a more conceptual argument, as follows. The configuration spaces \(C_n(\mathbb{C})\) have a well-known (real) Fulton-MacPherson compactification \(\overline{C}_n(\mathbb{C})\), by systematically adding strata corresponding to collapses of points labeled by subsets \(S\subset [n]\). These compactifications do not change the cohomology, because a manifold with boundary is always homotopy equivalent to its interior, but have the nice feature of making \(\overline{C}(\mathbb{C})\) into a (symmetric) operad. One can instead choose to only add those strata that correspond to consecutive subsets, and obtain partial compactifications \(C^{\delta}_n(\mathbb{C})\). These will also have the same cohomology as the uncompactified space, but can only be assembled to a nonsymmetric operad. Moreover, the projections \(C_n(\mathbb{C})\to M_{0,n+1}\) can be extended to the boundary strata to define projections \(C^{\delta}_n(\mathbb{C})\to M^{\delta}_{0,n+1}\), defining a morphism of nonsymmetric operads. Pullback along these projections is the suggested morphism of cooperads.
\end{proof}
\begin{corollary}
There is a canonical morphism of algebras \(\mathfrak{Z}_{\scriptstyle{\mathrm{ns}}}(A(\mathfrak{M}^{\delta}_0))\to \mathfrak{Z}_{\scriptstyle{\mathrm{ns}}}(\mathsf{coGer})\). 
\end{corollary}
We will prove that this morphism is a surjection \(Q\boldsymbol{\mathcal{Z}}_{\scriptstyle{\mathrm{ns}}}(A(M^{\delta}_0)) \to \boldsymbol{\mathcal{Z}}_{\scriptstyle{\mathrm{ns}}}(\mathsf{coGer})\) on indecomposable cohomological weights. To do this we first prove some structural results about the Arnol'd algebra.
\begin{definition}
First, for an ordered set \(S\), let \(\tilde{G}(S)^k\) denote the set of sets \(\{(i_1,j_j),\dots,(i_k,j_k)\}\) of \(k\) pairs \((i_r,j_r)\) of elements in \(S\) satisfying \(i_r<j_r\in S\) and such that no two pairs are equal. We refer to \(\tilde{G}(S)^k\) as the set of length \(k\) monomials. Then, define \(G(S)^k\) to be the subset consisting of those monomials \(M\in\tilde{G}(S)^k\) with the property that there are no two \((i,j), (j,k)\in M\). Let \(\tilde{G}(S):=\bigcup_{k=1}^{n-1}G(S)^k\), for \(n:=\#S\), and define \(G(S)\) analogously.
\end{definition}
\begin{lemma}\label{arnoldbasis}
The function \(\omega:G(n)^k\to \mathsf{coGer}(n)^k\) that sends \(\{(i_1,j_j),\dots,(i_k,j_k)\}\) to the monomial \(\omega_{i_1j_1}\dots\omega_{i_kj_k}\) identifies \(G(n)^k\) as a basis of \(\mathsf{coGer}(n)^k\).
\end{lemma}

\begin{proof}
The proof is parallel to the construction of a basis for \(\mathsf{Lie}(n)\), given in \cite{SalvatoreTauraso09} and is based on repeated use of the Arnol'd relation. First of all, it is clear that \(\omega(\tilde{G}(n)^k)\) spans the degree \(k\) summand of the Arnol'd algebra. Say that \((i,j)\) is path-connected of length \(q+1\) in a monomial \(M\) if there are \((i,r_1),(r_1,r_2),\dots ,(r_q,j)\in M\), and say that \((i,j)\) has index \(p\) in a monomial \(M\in G(n)\) if \((i,j)\) is path-connected in \(M\) and there are exactly \(p+1\) other pairs \((r,s)\in M\) such that \(r\leq i\) and \(j\leq s\). For example, if \((1,n)\in M\), then it must necessarily have index \(1\). This allows us to put a decreasing filtration on the Arnol'd algebra, by letting \(F^p\mathsf{coGer}(n)\) be spanned by monomials \(\omega(M)\) with the property that for all connected \((i,j)\) in \(M\) of index \(p\), the restriction of \(M\) to \(\tilde{G}(\{i,\dots ,j\})\) lies in \(G(\{i,\dots ,j\})\). Then \(F^0\mathsf{coGer}(n)=\mathsf{coGer}(n)\). We claim that the Arnol'd relation implies \(F^1\mathsf{coGer}(n)=\mathsf{coGer}(n)\). To see this, note that
 \[
  \omega_{ir_1}\omega_{r_1r_2}\dots\omega_{r_qj}=(\omega_{ir_1}\omega_{ir_2}-\omega_{r_1r_2}\omega_{ir_2})\omega_{r_2r_3}\dots\omega_{r_qj}.
 \]
After this \((i,j)\) is path-connected of length \(q-1\), and the restriction to \(\{i,\dots , r_2\}\) will lie in \(G(\{i,\dots , r_2\})\). Iterating we can reduce to a monomial in \(G(\{i,\dots ,j\})\). Arguing inductively on the filtration degree \(p\), we conclude that the monomials in \(G(n)\) span the Arnol'd algebra.
\end{proof}

\begin{definition}
Define \(L(n)\) to be the set of iterated formal binary bracketings of the indicies \(1,\dots , n\), subject to the following conditions:
 \begin{itemize}
 \item[-] Each index appears exactly once. (Thus the word must be an iteration of \(n-1\) binary brackets.)
 \item[-] The smallest index in a bracket stands to the left and the largest to the right.
 \end{itemize}
\end{definition}
For example, \([1,[2,3]]\) and \([[1,2],3]\) both lie in \(L(3)\), but neither \([2,[1,3]]\) nor \([[1,3],2]\) does. For each \(L\in L(n)\), define an ordering on the bracketings in \(L\) by reading them outside in and left to right. For example, the first bracket in \([[1,3],[2,4]]\) is that between \([1,3]\) and \([2,4]\), the second that between \(1\) and \(3\) and the third that between \(2\) and \(4\). To each bracket, associate the pair \((i,j)\), where \(i\) is the smallest index appearing in the bracket and \(j\) is the largest. In this way, we associate to each \(L\) a monomial \(M_L=\{(i_1,j_r),\dots ,(i_{n-1},j_{n-1})\}\). For example, if \(L=[[1,3],[2,4]]\), then \(M_L=\{(1,4),(1,3),(2,4)\}\).
\begin{remark}
The association \(L\mapsto M_L\) is a bijection from \(L(n)\) to \(G(n)^{n-1}\). The map is clearly injective and a cardinality count implies that it must be surjective. This witnesses the fact that \(\mathsf{coGer}(n)^{n-1}\cong \mathsf{coLie}(n)\).
\end{remark}
Identify \(C_n(\mathbb{C})\) with the subspace of \(\mathbb{C}^n\) consisting of all \((z_i)_{i=1}^n\) such that \(z_1=0\), \(\vert z_n\vert=1\) and \(z_i\neq z_j\) if the indicies are different. The projection \(C_n(\mathbb{C})\to M_{0,n+1}\) has a section which can be decribed by identifying \(M_{0,n+1}\) with the subspace of \(C_n(\mathbb{C})\) consisting of those \(n\)-tuples that in addition satisfy \(z_n=1\). With these identifications we obtain a description of \(H(M_{0,n+1})\) as the subalgebra of the Arnol'd algebra spanned by all \(\omega_{ij}\) except \(\omega_{1n}\), so, in effect
 \[
  \mathsf{coGer}=H(M_{0,n+1})[\omega_1n]=H(M_{0,n+1})\oplus \omega_{1n}H(M_{0,n+1}).
 \]
\begin{remark}
It was shown by Ezra Getzler in \cite{Getzler95} that \(H^{n-2}(M_{0,n+1})\cong \mathsf{coLie}(n)\), and this isomorphism now takes the following form: For every \(L\in L(n)\), the Arnol'd form \(\omega_L:=\omega(M_L)\) is of the form \(\omega_{1n}\alpha_L\). Thus we obtain an isomorphism \(\alpha: L(n)\to H^{n-1}(M_{0,n+1}), L\to \alpha_L\). 
\end{remark}
\begin{definition}
Say that a binary bracket \(b\) (of bracketings) in an \(L\in L(n)\) is \textit{connected} if the set of indices appearing inside \(b\) is a connected subset of \([n]\). Define the set of prime bracketings, to be denoted \(P(n)\), to be the subset of \(L(n)\) consisting of all those \(P\) with the property that only the outermost bracket is connected.
\end{definition}
\begin{lemma}\cite{SalvatoreTauraso09}
The operad \(\mathsf{Lie}\) is freely generated as a nonsymmetric operad by the collection \(\{P(n)\}_{n\geq 2}\).
\end{lemma}
\begin{corollary}
The association \(\alpha:L(n)\to H(M_{0,n+1})\) of forms on the moduli space to Lie words restricts to an isomorphism \(\alpha: P(n)\to A(M^{\delta}_{0,n+1})^{n-2}\).
\end{corollary}

\begin{proof}
Getzler proved in \cite{Getzler95} that the isomorphisms \(H^{n-1}(M_{0,n+1})\cong \mathsf{coLie}(n)\) is one of cooperads, where the cooperadic cocomposition maps on \(H(M_0)[-1]\) are given by residue along the respective boundary divisors of the Deligne-Mumford compactification \(\overline{M}_0\). Brown's moduli spaces sit inside the Deligne-Mumford compactification as the partial compactification given by adding only the strata corresponding to nonsymmetric (co)compositions. Thus, \(A(M^{\delta}_{0,n+1})\) can be defined as the intersection of the kernels of all the cooperadic cocompositions \(H(M_{0,n+1})\to H(M_{0,n-k+2})\otimes H(M_{0,k+1})\). This says dually that the dual space of \(A(M^{\delta}_{0,n+1})\) is isomorphic to the cokernel of all the nonsymmetric operadic compositions that land in \(\mathsf{Lie}(n)\). Hence \(P(n)\) must be a basis for \(A(M^{\delta}_{0,n+1})\). That it has the explicit form given by \(\alpha\) boils down to proving that \(\alpha:L(n)\to H(M_{0,n+1})\) defines a morphism of nonsymmetric cooperads \(\mathsf{coLie}\to H(M_0)[-1]\), when the latter is identified with the subalgebra of the Arnol'd algebra spanned by all \(\omega_{ij}\) except \(\omega_{1n}\). To argue this, we note that the relevant nonsymmetric cocompositions (given by residue) of \(H(M_0)[-1]\) have the following form. Assume that the points labelled by a connected subset \(\{i,\dots,i+k\}\subset n\) collapse. Take some \(\alpha_L\), \(L\in L(n)\). Let \(\Delta\) be the cocomposition of \(\mathsf{coGer}\) corresponding to the collapse. The cocomposition of \(H(M_{0,n+1})\) is then given by
 \[
  \bigl(\frac{\partial}{\partial \omega_{1\,n-k+1}}\otimes \frac{\partial}{\partial \omega_{1k}}\bigr)\Delta(\omega_{1n}\alpha_L).
 \]
Recall \(\omega_L=\omega_{1n}\alpha_L\). We then use that \(\omega_L\) is a representative of the (co)Lie word \(L\) under the isomorphism \(\mathsf{coLie}(n)\cong\mathsf{coGer}(n)\), which is an isomorphism of cooperads, to conclude that the above equals
 \[
  \bigl(\frac{\partial}{\partial \omega_{1\,n-k+1}}\otimes \frac{\partial}{\partial \omega_{1k}}\bigr)\omega_{\Delta L},
 \]
the \(\Delta\) now referring to the cocomposition on \(\mathsf{coLie}\). In terms of bracketings, we can understand \(\alpha_L\) as defined by the same combinatorial rule as \(\omega\), except we first remove the outermost bracket of \(L\). Applying the partial derivatives above does the same thing, for the boundary restrictions of \(\omega_L\). Thus the above equals \(\alpha_{\Delta L}\).
\end{proof}
\begin{theorem}
The morphism \(Q\boldsymbol{\mathcal{Z}}_{\scriptstyle{\mathrm{ns}}}(A(M^{\delta}_0)) \to Q\boldsymbol{\mathcal{Z}}_{\scriptstyle{\mathrm{ns}}}(\mathsf{coGer})\) is surjective.
\end{theorem}
\begin{proof}
We will argue that if \(I(\alpha)\in \boldsymbol{\mathcal{Z}}_{\scriptstyle{\mathrm{ns}}}(\mathsf{coGer})^0\) is not defined by an \(\alpha\in\mathsf{coGer}(n)^{n-2}\) of the form \(\alpha=\alpha_P\) for a prime bracketing \(P\in P(n)\), then it is zero in \(Q\boldsymbol{\mathcal{Z}}_{\scriptstyle{\mathrm{ns}}}(\mathsf{coGer})\). We shall first consider the case when \(n\geq 4\), as \(n=3\) turns out to be somewhat execptional. 

First assume that \(\alpha\) contains a factor \(\omega_{1n}\). Then we can write \(\alpha=\omega_{1n}\beta\) where \(\beta\) is either a form depending only on \(z_1,\dots, z_{n-1}\) or on \(z_2,\dots,z_n\). In the first case we have
 \[
  \alpha = \pm \beta \,\sha_{\{1,\dots,n-1\},\{1,n-1,n\}}\,\omega_{13},\; \omega_{13}\in\mathsf{coGer}(3),
 \]
and in the second case \(\alpha=\pm \beta \,\sha_{\{2,\dots,n\},\{1,2,n\}}\,\omega_{13}\), \(\omega_{13}\in\mathsf{coGer}(3)\). This settles the case for all forms \(\alpha\) divisible by \(\omega_{1n}\), because it shows \(I(\alpha)=0\) already in \(Q\mathfrak{Z}_{\scriptstyle{\mathrm{ns}}}(A(M^{\delta}_0))\). Consider next the case when \(\alpha\) is not divisible by \(\omega_{1n}\). We may then assume that \(\alpha=\alpha_L\) for a (co)Lie word \(L\in L(n)\). If \(L\) is a prime bracketing, then we are done. If it is not, then it contains a bracketing \(b\) that encloses a consecutive subset \(\{i,\dots,i+k-1\}\), corresponding to an operadic composition \(L_1\circ_i L_2 = L\). We can then write \(\alpha = \pm \alpha_{L_1}\omega_{1\,i+k-1}\omega_{L_2}\) and deduce
 \[
  \alpha = \pm  \alpha_{L_1}\omega_{i\,i+1} \,\sha_{\{1,\dots,i,i+k-1,\dots n\},\{i,\dots,i+k-1\}}\,\alpha_{L_2}.
 \]
We have now shown that if \(n\geq 4\) and \(\alpha\in\mathsf{coGer}(n)^{n-2}\) is not of the form \(\alpha=\alpha_P\), then \(I(\alpha)\) decomposes as a notrivial product in \(\mathfrak{Z}_{\scriptstyle{\mathrm{ns}}}(\mathsf{coGer})^0\).

To finalize, assume \(n=3\). In this case \(\alpha\) must equal one of the forms \(\omega_{12},\omega_{23},\omega_{13}\in\mathsf{coGer}(3)\). We note that 
 \[
 dI(\omega_{12}\in\mathsf{coGer}(4))= I(\omega_{12})+I(\omega_{12})-I(\omega_{12})=I(\omega_{12}).
 \]
It follows that \(I(\omega_{12})=0\) in \(\boldsymbol{\mathcal{Z}}_{\scriptstyle{\mathrm{ns}}}(\mathsf{coGer})\). Analogously, \(I(\omega_{23})=dI(\omega_{34})\) and \(I(\omega_{13})= - dI(\omega_{14})\).
\end{proof}
The algebra \(\mathfrak{Z}_{\scriptstyle{\mathrm{ns}}}(A(M^{\delta}_0))\) is in many ways nicer than the algebra \(\mathfrak{Z}_{\scriptstyle{\mathrm{ns}}}(\mathsf{coGer})\). To begin with, \(A(M^{\delta}_{0,n+1})\) is concentrated in degrees \(\leq n-2\), so the degree truncations \(\tau^{\leq -1}\) in our definitions are superflous. Secondly, it contains a family of elements that in a clear sense correspond to multiple zeta values.

Let \(w=0^{k_r-1}1\dots 0^{k_1-1} 1\) be a word in the letters \(0\) and \(1\), where we assume \(k_r\geq 2\). Write the words as \(w=\epsilon_1\dots \epsilon_{\ell}\), \(\rho_i\in \{0,1\}\). The multiple zeta value \(\zeta(k_1,\dots,k_r)\) is the real number
 \[
 \zeta(k_1,\dots, k_r) := (-1)^r\int_{0<t_1<\dots < t_{\ell}< 1} \wedge_{k=1}^{\ell} \rho_{\epsilon_k}(t_k),
 \]
where \(\rho_0(t):=d\log t\) and \(\rho_1(t):=d\log(1-t)\). Each such integral can be written as an integral on Brown's moduli spaces, of a form in \(A(M^{\delta}_0)\), as follows. Like before, use the gauge freedom to identify
 \[
 M_{0,n+1} = \{ (t_k)_{1\leq k\leq \ell} \in \mathbb{C}^n \mid t_i\neq t_j\;\mathrm{if}\;i\neq j,  t_k\neq 0,1\},
 \; \ell=n-2.
 \]
These coordinates \(t_k\) are related to the coordinates \(u_{ij}\) by
 \[
 1-t_k = u_{1n}u_{2n}\dots u_{kn}, \; 
 t_k = u_{k+1\,n+1} u_{k+2\,n+1}\dots u_{n-1\,n+1}, \; 1\leq k\leq \ell.
 \]
Using this, we can define a top-dimensional form
 \[
 \alpha(k_1,\dots,k_r):=(-1)^r \wedge_{k=1}^{\ell} \rho_{\epsilon_k}(t_k)\in A(M^{\delta}_{0,n+1})^{n-2},
 \]
and a corresponding degree zero element
 \[
 I(k_1,\dots,k_r):=I(\alpha(k_1,\dots,k_r))\in \mathfrak{Z}_{\scriptstyle{\mathrm{ns}}}(A(M^{\delta}_0)).
 \]
These are defined such that under the evaluation \(\mathfrak{Z}_{\scriptstyle{\mathrm{ns}}}(A(M^{\delta}_0))\to\mathbb{R}\) defined by integration on the embedded associahedra, \(I(k_1,\dots,k_r)\) is mapped to the multiple zeta value \(\zeta(k_1,\dots,k_r)\). 
\begin{remark}
The (image of the) generator \(x^{01}\) corresponds to the multiple zeta value \(\zeta(2)\), or to \(I(2)\). The tetrahedron \(\sigma_3\in\mathfrak{grt}_1\) corresponds to \(I(3)\).
\end{remark}
We end by making some conjectures.
\begin{conjecture}
There is a formal version \(\Phi\in \boldsymbol{\mathcal{Z}}_{\scriptstyle{\mathrm{ns}}}(A(M^{\delta}_0))\langle\langle x_0,x_1\rangle\rangle\) of the Knizhnik-Zamolodchikov Drinfel'd associator; having the \(I(k_1,\dots,k_r)\)'s as coefficients rather than actual multiple zeta values. Note that this would define a morphism from \(\mathfrak{grt}_1'\oplus\mathbb{Q}x^{01}\) to \(Q\boldsymbol{\mathcal{Z}}_{\scriptstyle{\mathrm{ns}}}(A(M^{\delta}_0))\). From \cite{Furusho11} it would then also follow that the \(I(k_1,\dots,k_r)\)'s satisfy the double shuffle relations. 
\end{conjecture}
\begin{conjecture}
The formal weights \(I(k_1,\dots,k_r)\) generate \(\boldsymbol{\mathcal{Z}}_{\scriptstyle{\mathrm{ns}}}(A(M^{\delta}_0))\). Combined with the previous conjecture this would imply that the morphism from \(\mathfrak{grt}_1'\oplus\mathbb{Q}x^{01}\) to \(Q\boldsymbol{\mathcal{Z}}_{\scriptstyle{\mathrm{ns}}}(A(M^{\delta}_0))\) defined by the formal Knizhnik-Zamolodchikov associator is onto.
\end{conjecture}
\begin{conjecture}
The indecomposable cohomological weights \(Q\boldsymbol{\mathcal{Z}}_{\scriptstyle{\mathrm{ns}}}(A(M^{\delta}_0))\) and \(Q\boldsymbol{\mathcal{Z}}_{\scriptstyle{\mathrm{ns}}}(\mathsf{coGer})\) are isomorphic, and both are isomorphic to \(\mathfrak{grt}_1'\oplus\mathbb{Q}x^{01}\).
\end{conjecture}
\bibliography{algebraofformalweights}{}
\bibliographystyle{plain}

\end{document}